\documentclass[12pt,letterpaper,reqno]{amsart}

\usepackage{times}
\usepackage[T1]{fontenc}
\usepackage{mathrsfs}
\usepackage{latexsym}
\usepackage[dvips]{graphics}
\usepackage{epsfig}
\usepackage{amsmath,amsfonts,amsthm,amssymb,amscd}
\input amssym.def
\input amssym.tex

\newcommand{\supp}{{\rm supp}}

\newcommand{\hphi}{\widehat{\phi}}  

\newcommand\be{\begin{equation}}
\newcommand\ee{\end{equation}}
\newcommand\bea{\begin{eqnarray}}
\newcommand\eea{\end{eqnarray}}
\newcommand\bi{\begin{itemize}}
\newcommand\ei{\end{itemize}}
\newcommand\ben{\begin{enumerate}}
\newcommand\een{\end{enumerate}}
\newcommand\bc{\begin{center}}
\newcommand\ec{\end{center}}
\newcommand\ba{\begin{array}}
\newcommand\ea{\end{array}}

\def\notdiv{\ \mathbin{\mkern-8mu|\!\!\!\smallsetminus}}

\newcommand{\R}{\ensuremath{\mathbb{R}}}

\newcommand{\Q}{\mathbb{Q}}

\newcommand{\foh}{\frac{1}{2}}  

\newtheorem{thm}{Theorem}[section]

\newtheorem{lem}[thm]{Lemma}

\newtheorem{rek}[thm]{Remark}

\newcommand{\twocase}[5]{#1 \begin{cases} #2 & \text{{\rm #3}}\\ #4
&\text{{\rm #5}} \end{cases}   }

\newcommand{\fourcase}[9]{#1 \begin{cases} #2 & \text{{\rm #3}}\\ #4
&\text{{\rm #5}}\\ #6 & \text{{\rm #7}}\\ #8 & \text{{\rm #9}} \end{cases}   }

\newcommand{\da}{\downarrow}

\newcommand{\gep}{\epsilon}
\newcommand{\gl}{\lambda}

\numberwithin{equation}{section}

\begin{document}

\title{An Orthogonal Test of the $L$-functions Ratios Conjecture, II}

\author[Miller]{Steven J. Miller}\email{Steven.J.Miller@williams.edu}
\address{Department of Mathematics and Statistics, Williams College, Williamstown, MA 01267}

\author[Montague]{David Montague}\email{davmont@umich.edu}
\address{Department of Mathematics, University of Michigan, Ann Arbor, MI 48109}


\date{\today}

\thanks{This work was done at the 2009 SMALL Undergraduate Research Project at Williams College, funded by NSF Grant DMS-0850577 and Williams College; it is a pleasure to thank them and the other participants, as well as Brian Conrey, David Farmer, David Hansen, Wenzhi Luo, and Peter Sarnak for comments on an earlier draft. The first named author was also partly supported by NSF Grant DMS-0855257.}

\begin{abstract} Recently Conrey, Farmer, and Zirnbauer \cite{CFZ1,CFZ2} developed the $L$-functions Ratios conjecture, which gives a recipe that predicts a wealth of statistics, from moments to spacings between
adjacent zeros and values of $L$-functions. The problem with this method is that several of its steps involve ignoring error terms of size comparable to the main term; amazingly, the errors seem to cancel and the
resulting prediction is expected to be accurate up to square-root cancellation. We prove the accuracy of the Ratios Conjecture's prediction for the 1-level density of families of cuspidal
newforms of constant sign (up to square-root agreement for support in $(-1,1)$, and up to a power savings in $(-2,2)$), and discuss the arithmetic significance of the lower order terms. This is the most involved test of the Ratios Conjecture's predictions to date, as it is known that the error terms dropped in some of the steps do \emph{not} cancel, but rather contribute a main term! Specifically, these are the non-diagonal terms in the Petersson formula, which lead to a Bessel-Kloosterman sum which contributes only when the support of the Fourier transform of the test function exceeds $(-1, 1)$.

\end{abstract}

\maketitle

\tableofcontents

\section{Introduction}

\subsection{Background}

The $L$-functions Ratios Conjecture of Conrey, Farmer, and Zirnbauer \cite{CFZ1, CFZ2} has been a very strong predictive tool for computing statistics related to a wide variety of families of $L$-functions. The conjecture is essentially a general recipe for averaging the values of ratios of $L$-functions over a family. These averages can then be used to predict the answers to deep questions about the distribution of zeros and values of the $L$-functions.

The Ratios Conjecture has been able to very accurately predict a wealth of statistics related to families of $L$-functions, ranging from n-level correlations and densities to mollifiers and moments to vanishing at the
central point \cite{CS1,CS2,GJMMNPP,HuyMil,Mil3,Mil5,St}. One reason the conjecture is so useful is that it usually gives its conjectured answer within a few pages of largely straightforward calculations, as opposed to the in-depth and lengthy analysis often required to make unconditional statements about these statistics (e.g. \cite{ILS}). Moreover, the high degree of accuracy -- the Ratios Conjecture is expected to be accurate down to square-root cancelation -- allows us to isolate any significant lower order terms.

These lower order terms are of interest for several reasons. For example, the main term of these statistics is often independent of the arithmetic of the family. While Random Matrix Theory has successfully predicted these values, it misses arithmetic,\footnote{There are now many families where the main term of the 1-level density agrees with the random matrix predictions and the lower order terms differ due to arithmetic features of the families; see \cite{FI,Mil2,Mil4,MilPe,Yo1}.} which frequently has to be added in a somewhat ad hoc manner.\footnote{For another approach to modeling $L$-functions which incorporates arithmetic, see the hybrid model of Gonek, Hughes and Keating \cite{GHK}.} The Ratios Conjecture has the arithmetic of the family enter in a natural way, and its presence is felt in the lower order terms. These terms are important in studying finer convergence questions.\footnote{\label{footnote:effective} For example, at first the zeros of $L$-functions high on the critical line were modeled by  the $N\to\infty$ scaling limits of $N\times N$ complex Hermitian matrices. Keating and Snaith \cite{KeSn1,KeSn2} showed that a better model for zeros at height $T$ is given by $N\times N$ matrices with $N \sim \log(T/2\pi)$; we use this for $N$ as it makes the mean spacing between zeros and eigenvalues equal. Even better agreement (see \cite{BBLM}) has been found by replacing $N$ with $N_{\rm effective}$, where the first order correction terms are used to slightly adjust the size of the matrix (as $N\to\infty$, $N_{\rm effective}/N \to 1$).} Additionally, the Ratios Conjecture also suggests alternate ways of writing the lower order terms, and these formulations often clarify the cause of these corrections. One instance is in the lower order terms of the family of quadratic Dirichlet characters, where one of the correction terms is seen to arise from the imaginary parts of zeros of $\zeta(s)$ (see \cite{Mil3,St}).

In this paper, which is a sequel to \cite{Mil5}, we investigate families of cuspidal newforms split by sign of the functional equation. We first set some notation; see \cite{IK,ILS} for more details and proofs. Let $f \in S_k(N)$, the space of cusp forms of weight $k$ and level $N$, let $\mathcal{B}_k(N)$ be an orthogonal basis of $S_k(N)$, and let $H^\star_k(N)$ be the subset of newforms. To each $f$ we associate an $L$-function
\begin{equation}
L(s,f)\ =\ \sum_{n=1}^\infty \lambda_f(n) n^{-s}
\end{equation}
with completed $L$-function
\begin{equation}\label{eq:completed_L_func}
\Lambda(s,f) \ =\ \left(\frac{\sqrt{N}}{2\pi}\right)^s
\Gamma\left(s+\frac{k-1}{2}\right) L(s,f)  \ = \ \epsilon_f \Lambda(1-s,f),
\end{equation}
with $\epsilon_f = \pm 1$. The space
$H^\star_k(N)$ splits into two disjoint subsets, $H^+_k(N) = \{
f\in H^\star_k(N): \epsilon_f = +1\}$ and $H^-_k(N) = \{ f\in
H^\star_k(N): \epsilon_f = -1\}$. From Equation $2.73$ of \cite{ILS} we have for $N > 1$ that
\begin{equation}\label{eq:number of terms in hkpm}
|H_k^\pm(N)| \ = \ \frac{k-1}{24}N + O\left( (kN)^{5/6}
\right);
\end{equation} thus a power savings in terms of the cardinality of the family will mean errors of size $O(N^{1/2})$.
We often assume the Generalized Riemann Hypothesis (GRH), namely that all non-trivial zeros of $L(s,f)$ have real part $1/2$.

In this paper, we determine the $L$-functions Ratios Conjecture's prediction for the 1-level density for the family $H_k^\pm(N)$, with $k$ fixed and $N \to \infty$ through the primes, and we show that it agrees with number theory for suitably restricted test functions. Recall the 1-level density for a family $\mathcal{F}$ of $L$-functions is
\begin{eqnarray}
D_{1,\mathcal{F}}(\phi)\ :=\ \frac{1}{|\mathcal{F}|} \sum_{f\in
\mathcal{F}} \sum_{\ell} \phi\left(\gamma_{f,\ell}\frac{\log
Q_f}{2\pi}\right),
\end{eqnarray} where $\phi$ is an even Schwartz test function
whose Fourier transform has compact support, $\foh +
i\gamma_{f,\ell}$ runs through the non-trivial zeros of $L(s,f)$ (if GRH holds, then each $\gamma_{f,\ell} \in \R$), and $Q_f$ is the analytic conductor of $f$. As $\phi$ is an even
Schwartz functions, most of the contribution to $D_{1,\mathcal{F}}(\phi)$ arises from the zeros near the central point;\footnote{This statistic is very different than the $n$-level correlations, where we may remove arbitrarily many zeros without changing the limiting behavior. Knowing all the $n$-level correlations would give us the spacing statistics between adjacent zeros. To date we know these correlations for suitably restricted test functions for $L$-functions arising from cuspidal automorphic representations of ${\rm GL}_m/\Q$ if $m \le 3$ (and in general under additional hypotheses, such as the general Ramanujan conjectures for cusp forms on ${\rm GL}_m$). See \cite{Hej,Mon,RS,Od1,Od2} for results on $n$-level correlations and comparison of spacings between zeros and random matrix predictions.} thus, this statistic is well-suited to investigating the low-lying zeros (the zeros near the central point). Katz and Sarnak have conjectured that each family of $L$-functions  corresponds to some classical compact group which determines many properties and statistics related to the family. Specifically, for an infinite family of $L$-functions let $\mathcal{F}_N$ be the sub-family whose conductors either equal or are at most $N$. They conjecture that \be \lim_{N\to\infty} D_{\mathcal{F}_N}(\phi) \to \int \phi(x) W_{G(\mathcal{F})}(x)dx,\ee where $G(\mathcal{F})$ indicates unitary, symplectic or orthogonal (possibly ${\rm SO(even)}$ or ${\rm SO(odd)}$) symmetry; this has been observed in numerous families, including all Dirichlet characters, quadratic Dirichlet characters, $L(s,\psi)$ with $\psi$ a character of the ideal class group of the imaginary quadratic field $\mathbb{Q}(\sqrt{-D})$ (as well as more general number fields), families of elliptic curves, weight $k$ level $N$ cuspidal newforms, symmetric powers of ${\rm GL}(2)$ $L$-functions, and certain families of ${\rm GL}(4)$ and ${\rm GL}(6)$ $L$-functions (see \cite{DM1,DM2,FI,Gu,HR,HuMil,ILS,KaSa2,Mil1,MilPe,OS2,RR,Ro,Rub1,Yo2}).

We briefly summarize what is done in this paper. In the next subsection we describe the Ratios Conjecture's recipe to predict the $1$-level density for a family. We state our main results in \S\ref{sec:mainresults}, and then discuss in the next subsection why this is such an important test of the Ratios Conjecture, perhaps the most delicate one to date. We begin the main part of the paper by following the Ratios Conjecture's recipe for the family of cuspidal newforms of weight $k$ and level $N$ as $N$ tends to infinity through the primes, and determine the predicted $1$-level density for this family. We then use the Ratios Conjecture's prediction to isolate lower order terms in the $1$-level density. Finally, in \S\ref{sec:THEORY}, we elaborate on computations from \cite{ILS} to show strong agreement between theory and the Ratios Conjecture (see Theorem \ref{thm:sqrt}), which validates (for suitably restricted test functions) the computation of the lower order terms.

\subsection{The Ratio Conjecture's Recipe}

For a given family of $L$-functions $\mathcal{F}$, we are interested in estimating the quantity
\be R_\mathcal{F} (\alpha, \gamma) := \sum_{f \in \mathcal{F}} \omega_f \frac{L\left(\foh+\alpha, f\right)}{L\left(\foh+\gamma, f\right)},
\ee where the $\omega_f$ are weights specific to the family. We use this estimate to determine other statistics related to the zeros of the $L$-functions in the family of interest. To determine the $L$-functions Ratios Conjecture's prediction for this quantity, we follow several steps. We describe the recipe in general, highlighting how we apply it for our family. See \cite{CS1} for an excellent description of how to use the conjecture for a variety of problems. \\

\ben
\item We begin by using the approximate functional equation to expand the numerator $L$-function, giving two sums and an error term. In the approximate functional equation, the first sum is up to $x$, and the second is up to $y$, where $xy$ is of the size of the analytic conductor of $L(s,f)$. In following the Ratios Conjecture, we ignore the error term. As our family is cuspidal newforms of weight $k$ and level $N$, the approximate functional equation reads (see \cite{IK} for a proof) \be L(s,f)\ =\ \sum_{m \le x} \frac{a_m}{m^s} + \epsilon X_L(s) \sum_{n \le y} \frac{a_n}{n^{1-s}} + R(s,f),\ee where $R(s,f)$ denotes a remainder term (which we ignore in following the Ratios Conjecture), and $X_L$ (related to the functional equation for $L(s,f)$) is \be\label{ln:XL} X_L(s)\ =\ \left(\frac{\sqrt{N}}{2 \pi}\right)^{1 - 2s} \frac{\Gamma\left(1-s + \frac{k - 1}{2}\right)}{\Gamma\left(s + \frac{k - 1}{2}\right)}. \ee Note that $X_L(s)$ only depends weakly on $f$, as it is a function only of the level $N$ and the weight $k$.\\

\item Next, we expand the denominator $L$-function through its Dirichlet series via the generalized Mobius function $\mu_f$, where
\be \frac{1}{L(s,f)}\ =\ \sum_{h=1}^\infty \frac{\mu_f(h)}{h^s}. \ee For cuspidal newforms, $\mu_f(n)$ is the multiplicative function given by \be \fourcase{\mu_f(p^r) \ = \ }{1}{if $r=0$}{-\lambda_f(p)}{if $r=1$}{\chi_0(p)}{if $r=2$}{0}{if $r \ge 3$;} \ee here $\chi_0$ is the principal character modulo the level $N$ (so $\chi_0(p) = 1$ if $p \notdiv N$).\\

\item We now execute the sum over the family $\mathcal{F}$, using some averaging formula for the family in question. As we will be studying families of cuspidal newforms in this paper, we use the Petersson formula (see Appendix \ref{sec:PeterssonFormula} for statements). As part of the Ratios Conjecture, we drop all non-diagonal or non-main terms that arise in applying the averaging formula, and we ignore the error in doing so. The test performed in this paper is very important because the non-diagonal terms that are dropped are known to contribute a main term to the $1$-level density (see \cite{ILS}); however, we \emph{still} find agreement between theory and the $L$-functions Ratios Conjecture's prediction. We discuss this in great detail below.\\

\begin{rek} In the original formulation of the Ratios Conjecture, in Step 3 we are supposed to replace any products of signs of functional equations with their average value over the family. For families with constant sign of the functional equation, there is no difference. Even though our families are of constant sign, in our expansions above it is convenient to replace the summation over the family by sums over all cuspidal newforms of weight $k$ and level $N$ through factors such as $1 \pm \gep_f$, as this facilitates applying the Petersson formula. Following \cite{Mil5}, we consider a weaker version of the Ratios Conjecture where these terms are not dropped. The analysis is similar, and in Appendix \ref{ap:SFE} we see these terms (as predicted) do not contribute.\\
\end{rek}

\item After averaging over the family (which, in our case, is facilitated by the presence of the weights $\omega_f$), we extend the sums from the approximate functional equation to infinity. Often, we rewrite the sums as products before extending them, in which case this step is just completing the products.\\

\item In order to compute statistics related to the zeros, we typically differentiate the average with respect to the numerator $L$-function's parameter, and set both parameters ($\alpha$ and $\gamma$) equal. This gives an estimate for the logarithmic derivative of the $L$-functions averaged over the family. We note that thanks to Cauchy's integral formula, the size of the error term does not increase significantly when we differentiate (see Remark 2.2 of \cite{Mil5} for a proof).\\

\item The $1$-level density can be obtained by performing a contour integral of the differentiated average (which represents logarithmic derivative of $L(s,f)$ averaged over the family) from the previous step.\\

\een

\subsection{Main Results}\label{sec:mainresults}

We try to share notation with \cite{ILS,Mil5} as much as possible. The following infinite product arises several times in this paper and in \cite{ILS} (see their Section 7): \be \chi(s)
\ :=\ \prod_p \left(1 + \frac{1}{(p-1)p^s}\right)\ =\ \sum_{n=1}^\infty \frac{\mu^2(n)}{\varphi(n)n^s}. \ee Note the factorization given in \cite{ILS} is wrong; fortunately their factorization does give the correct main term, which is all that was studied there.

\begin{thm}\label{thm:ratios} For $R$ a constant multiple of $N$, the $L$-functions Ratios Conjecture predicts that the weighted, scaled 1-level density is equal to
\bea D_{1, H_k^{\pm} (N); R}(\phi) & = & \sum_{p} \frac{2\log p}{p\log R} \hphi \left(\frac{2\log p}{\log R}\right) + \frac{\log N}{\log R}\hphi(0) \nonumber \\
                                    & & \mp 2\lim_{\epsilon \downarrow 0} \int_{-\infty}^\infty  X_L\left(\foh + 2\pi i x\right) \chi(\epsilon + 4 \pi ix) \phi(t\log R) dt \nonumber \\
                                    & & + \frac{2}{\log R} \int_{-\infty}^{\infty}\frac{\Gamma'}{\Gamma}\left(\frac{k}{2}+\frac{2\pi it}{\log R}\right)\phi(t)dt + O(N^{-1/2+\epsilon}). \eea
\end{thm}

In \S\ref{sec:THEORY}, we confirm the prediction of Theorem \ref{thm:ratios} for suitably restricted $\phi$, as specified in the following theorem.

\begin{thm} \label{thm:sqrt}
Assume GRH for $\zeta(s)$, Dirichlet $L$-functions $L(s,\chi)$, and $L(s,f)$. For even Schwartz functions $\phi$ such that $\supp(\hphi) \subset (-\sigma, \sigma) \subseteq (-2,2)$, and for $R$ a constant multiple of $N$, the 1-level density $D_{1, H_k^{\pm} (N); R}(\phi)$ agrees with the Ratios Conjecture's prediction up to $O(N^{-1/2+\epsilon} + N^{\sigma/2-1+\epsilon})$.
\end{thm}

\begin{rek} Theorem \ref{thm:sqrt} shows that the $L$-functions Ratios Conjecture gives the correct prediction up to square root cancellation for $\supp(\hphi) \subseteq (-1,1)$, and up to a power savings for $\supp(\hphi) \subseteq (-2,2)$.
\end{rek}

Because the lower order terms in the 1-level density can be applied to several problems, we isolate these terms. The most important is the $1/\log R$ term, which is used to compute $N_{{\rm effective}}$ (see Footnote \ref{footnote:effective}). It is given by

\begin{thm}\label{thm:LOT} The $L$-functions Ratios Conjecture predicts that, for any fixed $\delta > 0$,
\bea D_{1, H_k^{\pm} (N); R}(\phi) & = & \frac{1}{\log R}\sum_{p} \frac{2\log p}{p} \hphi \left(\frac{2\log p}{\log R}\right) \mp \frac{1}{2}\phi(0) \nonumber \\
     & & \pm \int_{-\infty}^{\infty} \left(\frac{\sin \left(2\pi t \frac{\log \frac{N}{4 \pi^2}}{\log R}\right)}{2\pi t}\right)\phi(t) dt \nonumber \\
     & & \mp \frac{m}{\log R}\ \hphi\left( \frac{\log \frac{N}{4\pi^2}}{\log R}\right) \nonumber \\
     & & + \frac{\log N}{\log R}\hphi(0) + \frac{2}{\log R} \int_{-\infty}^{\infty}\frac{\Gamma'}{\Gamma}\left(\frac{k}{2}+\frac{2\pi it}{\log R}\right)\phi(t)dt \nonumber \\
     & & +\ O\left((\log R)^{-2(1-\delta)}\right), \eea
     where \be m = 2\gamma - 2\sum_p \frac{\log p}{p(p+1)} -4 \frac{\zeta'}{\zeta}(2)-2\frac{\Gamma'}{\Gamma}\left(\frac{k}{2}\right). \ee
In particular, let $\ell = \frac{\log \frac{N}{4 \pi^2}}{\log R}$ (note $\ell \sim 1$, as we take $R$ to be a constant multiple of $N$). Then, for $\phi$ satisfying $\supp(\hphi) \subseteq (-\ell,\ell)$, and for any $A > 0$,
\bea D_{1, H_k^{\pm} (N); R}(\phi) & = & \frac{1}{\log R}\sum_{p} \frac{2\log p}{p} \hphi \left(\frac{2\log p}{\log R}\right) + \frac{\log N}{\log R}\hphi(0) \nonumber \\
     & & + \frac{2}{\log R} \int_{-\infty}^{\infty}\frac{\Gamma'}{\Gamma}\left(\frac{k}{2}+\frac{2\pi it}{\log R}\right)\phi(t)dt + O\left(\frac{1}{\log^A R}\right).\ \ \  \eea
\end{thm}

We note that by Theorem \ref{thm:sqrt} (which assumes only GRH for $\zeta$, Dirichlet $L$-functions, and $L(s,f)$), the $L$-functions Ratios Conjecture's prediction from Theorem \ref{thm:LOT} can be proved to be accurate for any $\phi$ satisfying $\supp(\hphi) \subseteq(-2,2)$.

\begin{rek} While performing the analysis contained within this paper, the authors originally determined Theorem \ref{thm:LOT} as a prediction of the $L$-functions Ratios Conjecture. Using the $L$-functions Ratios Conjecture to determine the lower order terms was significantly less involved than showing agreement between the theory and the conjecture, as the Ratios argument avoided the difficult analysis of the Bessel-Kloosterman terms. This is an excellent example of the $L$-functions Ratios Conjecture being used to streamline the computation of quantities like the lower order terms in the 1-level density.
\end{rek}

\subsection{Discussion}\label{subsec:discussion}

In \cite{ILS}, the main term in the 1-level densities for $H_k^\pm(N)$ was computed for test functions $\phi$, where $\supp(\hphi) \subset (-2,2)$. We extend these results by computing all lower order terms down to square-root cancelation in the family's cardinality. We first use the Ratios Conjecture to predict the answer, and then generalize the analysis in \cite{ILS} to show agreement. A similar test of the $L$-functions Ratios Conjecture was performed by Miller \cite{Mil5} for the family $H_k^*(N)$, where there is no splitting by sign of the functional equation. We briefly comment on why our test, namely splitting the family by the sign of the functional equation, is of significant interest.

In the analysis performed in \cite{ILS}, we see that the terms arising from splitting the family by the sign of the functional equation contribute equally and oppositely for opposite signs of the functional equation. For $\phi$ so that $\hphi$ is supported outside $(-1,1)$ but within $(-2,2)$, it is shown that the non-diagonal Bessel-Kloosterman sums (which arise from applying the Petersson formula) contribute a main term to the 1-level density; these terms \emph{did not} contribute a main term when $\supp(\hphi) \subset (-1,1)$. In other words, for small support these non-diagonal terms were not significant, and only became a main term as the support increased.

Because of this, we were concerned about the results from the third step in the Ratios Conjecture. That step involves dropping the non-diagonal terms, and from the analysis in \cite{ILS} we know that, in fact, the non-diagonal terms contribute a main term. This makes for a terrific test of the Ratios Conjecture -- significantly better than the test in \cite{Mil5} (as the test there did not split by sign of the funtional equation; the non-diagonal terms' contributions cancel each other out). We ultimately find, however, that the Ratios Conjecture ``knows'' about these non-diagonal terms, and is able to determine both the main term and lower order terms that arise in splitting the family by the sign of the functional equation. This phenomenal agreement was somewhat surprising\footnote{It is only somewhat surprising as the Ratios Conjecture's predictions have been shown to hold in numerous cases, which convinced us to have faith.}.

Another reason that this test of the Ratios Conjecture is so important is that it is a great example of the predictive philosophy of the Ratios Conjecture. The analysis of the non-diagonal Bessel-Kloosterman sums in \cite{ILS} is very involved and technical\footnote{In fact, when Hughes and Miller \cite{HuMil} study the $n$-level density (or $n$\textsuperscript{th} centered moments) of cuspidal newforms, they encounter a multi-dimensional analogue of these sums. To avoid having to evaluate these directly, they convert their sums to a one-dimensional Bessel-Kloosterman sum by changing variables, which leads to a new test function. The resulting answer looks very different from the Random Matrix Theory predictions, though, because RMT was expecting an $n$-dimensional integral to be evaluated. The two answers are shown to agree through combinatorics, which, though involved, are more pleasant than generalizing the results from \cite{ILS}. A nice offshoot of this analysis is a new formula for the $n$-level density which, for restricted support, is more convenient for comparisons with RMT than the determinantal formulas of Katz and Sarnak. Formulas such as these are useful, as it is not always easy to see that number theory and RMT agree (see for example Gao's thesis \cite{Gao}).}, and a great deal of effort must be put into determining their contribution. In contrast, we completely ignore these bothersome terms in the Ratios Conjecture analysis, and \emph{still} come to the same conclusion. In fact, most of the analysis on the Ratios Conjecture side of the computation is relatively standard, e.g. dealing with contour integrals (perhaps with a pole on the line of integration, at worst).

Finally, in the $1$-level density expansions, the Ratios Conjecture predicts a term involving the integral of $\phi(t)$ against an Euler product. In all other families studied to date \cite{GJMMNPP,Mil3,Mil5}, either there is no product term (as in the unitary family of Dirichlet characters), or the product term is of size $O(|\mathcal{F}_N|^{-1/2+\gep})$ (as in the family of quadratic Dirichlet characters or all cuspidal newforms). This family is the first time that the product, which depends on the arithmetic of the family, not only contributes significant lower order terms but also a main term; this is the first test where the arithmetic of the family has played such a large role.

%
%

\section{The Ratios Conjecture}

\subsection{Preliminaries}

In this paper, we are interested in verifying the $L$-Functions Ratios Conjecture by comparing the conjecture's prediction for the weighted 1-level density for the families $H_k^{\pm}(N)$ of $L$-functions for cuspidal newforms of weight $k$ and level $N$, with sign of the functional equation $\pm 1$.

The specific quantity we are interested in is: \be D_{1, H_k^{\pm} (N); R}(\phi) :=  \sum_{f \in H_k^{\pm}(N)} \omega_f^{\pm}(N) \sum_{\substack{\gamma_f \\ L(1/2+i\gamma_f,f) = 0}} \phi\left(\gamma_f\frac{\log R}{2\pi}\right), \ee where $\phi$ is an even Schwartz function whose fourier transform has finite support, and so can be analytically continued to an entire function.

We describe the weights $\omega_f^{\pm}(N)$. As in \cite{Mil5}, we need to investigate sums such as \be \sum_{f \in H_k^\ast(N)} \gl_f(m) \gl_f(n). \ee To avoid technical difficulties\footnote{In \cite{ILS} much work was done to remove these weights; following them and \cite{Mil5}, we may consider the unweighted sums as well. The unweighted sums are important for investigating bounds for order of vanishing at the central point; see \cite{HuMil}.}, we introduce weights, and instead consider \be \sum_{f \in H_k^\ast(N)} \omega_f(N) \gl_f(m) \gl_f(n), \ee where the $\omega_f(N)$ are the harmonic (or Petersson) weights. These are defined by \be \omega_f^\ast(N) \ =\ \frac{\Gamma(k-1)}{(4\pi)^{k-1} (f,f)_N},  \ee where \be (f,f)_N \ = \ \int_{\Gamma_0(N)\setminus \mathbb{H}} f(z) \overline{f}(z) y^{k-2} dxdy. \ee These weights are almost constant in that we have the bounds (see \cite{HL,Iw}) \be N^{-1-\gep}\ \ll_k \ \omega_f^\ast(N) \ \ll_k \ N^{-1+\gep}; \ee if we allow ineffective constants we can replace $N^\gep$ with $\log N$ for $N$ large.

The weights $\omega_f^{\pm}(N)$ are just twice the modified Petersson weights $\omega_f^*(N)$. We multiply them by a factor of two due to the fact that roughly half of the family $H_k^*(N)$ has odd, and roughly half has even sign of the functional equation, and so multiplying by two gives a better normalization of the weights. These weights simplify the Petersson formula (see Appendix \ref{sec:PeterssonFormula} for statements).

\begin{rek} Technically we should use the modified weights $\omega_f(N) / \omega(N)$, where $\omega(N) = \sum_{f \in H_k^\ast(N)} \omega_f(N)$, as we do not include the level 1 forms. As $N\to\infty$ and there are $O(1)$ such forms, this leads to an error of size $O(N^{-1+\gep})$, which is much smaller than our other error terms. Thus we may safely use these weights. See \S1.2 of \cite{Mil5} for a complete explanation of the choice of weights. \end{rek}

\subsection{The Ratios Conjecture's Prediction}

\begin{thm} \label{thm:RCP}
For $\Re(\alpha), \Re(\gamma) > 0$, the Ratios Conjecture predicts that \bea \mathcal{R}_{\pm}(N) & \ := \ & \sum_{f \in H_k^{\pm}(N)} \omega_f^{\pm} (N) \frac{L(\foh + \alpha, f)}{L(\foh + \gamma, f)}\nonumber\\ & = & \prod_{p} \left(1 - \frac{1}{p^{1+\alpha+\gamma}} + \frac{1}{p^{1+2\gamma}}\right) \pm X_L \left(\foh + \alpha \right) \nonumber \\ & & \cdot \ \ \frac{1}{\zeta(1-\alpha+\gamma)} \prod_{p} \left(1 + \frac{p^{1-\alpha+\gamma}}{p^{1 + 2\gamma}(p^{1-\alpha+\gamma}-1)}\right)  + O(N^{-1/2+\epsilon}).\ \ \  \eea
\end{thm}

\begin{proof}
In order to compute the 1-level density, we follow the steps in the Ratios Conjecture to determine:

\bea\label{eq:rcpeqtwoeight}
\mathcal{R}_{\pm}(N) & \ := \ & \sum_{f \in H_k^{\pm}(N)} \omega_f^{\pm} (N) \frac{L(\foh + \alpha, f)}{L(\foh + \gamma, f)}  \nonumber\\
& = & \sum_{f \in H_k^*(N)} (1 \pm \epsilon_f) \omega_f^* (N)  \left(\sum_{h = 1}^\infty \frac{\mu_f (h)}{h^{\foh + \gamma}}\right) \Bigg[ \sum_{m\le x} \frac{\lambda_f (m)}{m^{\foh + \alpha}} \nonumber \\
&   & +  \epsilon_f X_L \left(\foh + \alpha \right) \sum_{n \le y} \frac{{\lambda}_f (n)}{n^{\foh - \alpha}} \Bigg].
\eea

We now split this into two sums through the factor $(1 \pm \epsilon_f)$. Note that we use $(1 \pm \epsilon_f)$ instead of $(1 \pm \epsilon_f)/2$ because $\omega_f^{\pm} (N) = 2\omega_f^{*} (N)$. We assume $\Re (\alpha), \Re (\gamma) > 0$ wherever necessary, as this is the only region we need to consider.

\bea & & \mathcal{R}_{\pm}(N)  \ := \  \nonumber\\
& & \sum_{f \in H_k^*(N)} \omega_f^* (N) \sum_{h = 1}^\infty \frac{\mu_f (h)}{h^{\foh + \gamma}} \left [ \sum_{m\le x} \frac{\lambda_f (m)}{m^{\foh + \alpha}} + \epsilon_f X_L \left(\foh + \alpha\right) \sum_{n \le y} \frac{{\lambda}_f (n)}{n^{\foh - \alpha}}\right ] \nonumber \\
& \pm & \sum_{f \in H_k^*(N)} \omega_f^* (N) \sum_{h = 1}^\infty \frac{\mu_f (h)}{h^{\foh + \gamma}} \left[ \epsilon_f \sum_{m\le x} \frac{\lambda_f (m)}{m^{\foh + \alpha}} +  X_L \left(\foh + \alpha\right) \sum_{n \le y} \frac{{\lambda}_f (n)}{n^{\foh - \alpha}} \right]. \ \
\eea

Following the recipe of the Ratios Conjecture, we ignore terms involving the sign of the functional equation, as the sum is over $H_k^* (N)$, and for $N$ prime and greater than 1, the average sign of the functional equation is 0. We note that by an argument similar to that in \cite{Mil5}, it can be shown that both terms involving the sign of the functional equation here are $O\left(\frac{1}{N}\right)$, so we need not assume this strong of a version of the Ratios Conjecture (see Appendix \ref{ap:SFE} for more details). Thus, we define
\bea
S_1 & := & \sum_{f \in H_k^*(N)} \omega_f^* (N)  \sum_{h = 1}^\infty \frac{\mu_f (h)}{h^{\foh + \gamma}}\sum_{m\le x} \frac{\lambda_f (m)}{m^{\foh + \alpha}} \nonumber\\
S_2 & := & \pm \sum_{f \in H_k^*(N)} \omega_f^* (N)  \sum_{h = 1}^\infty \frac{\mu_f (h)}{h^{\foh + \gamma}} X_L \left(\foh + \alpha\right) \sum_{n \le y} \frac{{\lambda}_f (n)}{n^{\foh - \alpha}},
\eea
and so we are left to consider $S_1 + S_2$. Following the steps in \cite{Mil5}, we get
\bea S_1 & = & \prod_{p} \left(1 - \frac{1}{p^{1+\alpha+\gamma}} + \frac{1}{p^{1+2\gamma}}\right) \nonumber\\
				 &   & +\ O(N^{-1/2+\epsilon}) \nonumber\\
     S_2 & = & \pm X_L \left(\foh + \alpha \right) \frac{1}{\zeta(1-\alpha+\gamma)} \prod_{p} \left(1 + \frac{p^{1-\alpha+\gamma}}{p^{1 + 2\gamma}(p^{1-\alpha+\gamma}-1)}\right) \nonumber\\
         &   & +\ O(N^{-1/2+\epsilon}). \eea The computation for $S_1$ was done in \S2.2 of \cite{Mil5}; the computation of $S_2$ follows analogously.
\end{proof}

\begin{rek} The error terms arising above are added somewhat ad-hoc. They are only there because that is the level to which the $L$-functions Ratios Conjecture is expected to be accurate. \end{rek}

We now differentiate with respect to $\alpha$ to determine $\sum_{f\in H_k^{\pm}} \omega_f^* (N) \frac{L'(\foh + \alpha, f)}{L(\foh + \gamma, f)}$; note the differentiation does not increase the size of the error term (see Remark 2.2 of \cite{Mil5}). After determining this sum, we set $\alpha = \gamma = r$ to prepare for the contour integration to compute the predicted weighted 1-level density.

\begin{lem} \label{lem:S_1}
For $\Re(r) > 0$, the Ratios Conjecture predicts that \be \sum_{f \in H_k^{\pm}(N)} \omega_f^{\pm} (N) \frac{L'(\foh + r, f)}{L(\foh + r, f)}\ =\ \sum_{p} \left(\frac{\log p}{p^{1+2r}}\right) \mp X_L \left(\foh + r\right) \chi(2r) + O(N^{-1/2+\epsilon}), \ee where $\chi(s)$ is defined as \be \chi(s)\ :=\ \prod_{p} \left(1 + \frac{1}{(p-1)p^{s}}\right). \ee
\end{lem}

\begin{proof} First, we take advantage of the following expression for $\frac{d}{d\alpha}S_1(\alpha,\gamma)$:
\be \frac{dS_1(\alpha,\gamma)}{d\alpha}\Big |_{\alpha = \gamma = r}\ =\ S_1(\alpha,\gamma) \frac{d}{d\alpha}\log(S_1(\alpha,\gamma)) \Big |_{\alpha = \gamma = r}. \label{ln:ds} \ee
We now compute $\frac{d}{d\alpha}\log(S_1(\alpha,\gamma))$:
\bea
\frac{d}{d\alpha}\log(S_1(\alpha,\gamma)) &\ = \ &  \sum_{p} \frac{d}{d\alpha}\log(1 - \frac{1}{p^{1+\alpha + \gamma}} + \frac{1}{p^{1+2\gamma}}) \nonumber \\
					 &=& \sum_{p} \frac{-\frac{1}{p^{1+\alpha +\gamma}}(-\log p)}{1 - \frac{1}{p^{1+\alpha+\gamma}} + \frac{1}{p^{1+2\gamma}}} \nonumber \\
					 &=& \sum_{p} \frac{(\frac{\log p}{p^{1+\alpha +\gamma}})} {1 - \frac{1}{p^{1+\alpha +\gamma}} + \frac{1}{p^{1+2\gamma}}}.
\eea
With this, by equation \eqref{ln:ds} we have \be \frac{dS_1(\alpha,\gamma)}{d\alpha} \Big |_{\alpha = \gamma = r}\ =\ \prod_{p} 1 \sum_{p} \frac{\frac{\log p}{p^{1+2r}}}{1}\ =\ \sum_{p} \frac{\log p}{p^{1+2r}}. \ee
Next, \bea S_2 &\ =\ & \pm X_L \left(\foh + \alpha \right) \frac{1}{\zeta(1-\alpha+\gamma)} \prod_{p} \left(1 + \frac{p^{1-\alpha+\gamma}}{p^{1 + 2\gamma}(p^{1-\alpha+\gamma}-1)}\right) \nonumber \\
               & = & \frac{S_2^*(\alpha,\gamma)}{\zeta(1-\alpha+\gamma)}. \eea
We now use the following observation (see page 7 of \cite{CS1}). For a function $f(z,w)$ which is analytic at $(z,w) = (\alpha,\alpha)$, we have that \be \frac{d}{d\alpha}\frac{f(\alpha,\gamma)}{\zeta(1-\alpha+\gamma)} \Big|_{\gamma = \alpha}\ =\ -f(\alpha, \alpha). \ee Thus, we have that
\be \frac{dS_2}{d\alpha} \Big |_{\alpha = \gamma = r} = -S^*_2(r,r)\ =\ \mp X_L \left(\foh + r\right) \prod_{p} \left(1 + \frac{p}{(p-1)p^{1+2r}}\right). \ee
Summing the expression for the derivative of $S_1$ with that of $S_2$ gives the lemma.
\end{proof}


\subsection{Weighted 1-level density from the Ratios Conjecture}

We now evaluate a contour integral to determine $D_{1, H_k^{\pm} (N); R}(\phi)$.  We first calculate the unscaled 1-level density, written as $S_{1, H_k^{\pm} (N)} (g)$, where $g$ is related to $\phi$ by $g(t) = \phi\left(\frac{t\log R}{2\pi}\right)$. With this choice of $g$, a change of variables shows $D_{1, H_k^{\pm} (N); R}(\phi) = S_{1, H_k^{\pm} (N)} (g)$. Note that $S_{1, H_k^{\pm} (N)} (g)$ should not be confused with $S_1$ above (to which we will no longer refer). Let $c \in \left(\foh, \frac{3}{4}\right)$.
\bea
& & S_{1, H_k^{\pm} (N)} (g)\ :=\ \sum_{f\in H_k^{\pm} (N)} \omega_f^*(N) \sum_{\gamma_f} g(\gamma_f)  \nonumber \\
& =\ & \frac{1}{2\pi i}  \left(\int_{(c)} - \int_{(1-c)}\right) \sum_{f\in H_k^{\pm} (N)} \omega_f^* (N) \frac{L'}{L}(s,f) g\left(-i\left(s-\foh \right)\right) ds.
\eea
Because of its ultimate similarity to the integral over $\Re(s) = c$, we begin by considering the integral over $\Re(s) = 1-c$.

For ease of writing integrals, we introduce the following notation: let $G_+(s) = g\left(-i\left(s - \foh \right)\right)$, let $G_-(s) = g\left(-i\left(\foh - s\right)\right)$, and let $G(s) = G_+(s) + G_-(s)$. Note that $G_+(s) = G_-(1-s)$. Thus, we have
\bea \int_{(1-c)} &\ :=\ & \frac{1}{2\pi i} \int_{(1-c)} \left(\sum_{f\in H_k^{\pm} (N)} \omega_f^{\pm} (N) \frac{L'}{L}(s,f)\right)G_+(s)ds. \nonumber \\
& = & \frac{1}{2\pi i} \int_{-\infty}^{\infty}\Big[ \sum_{f\in H_k^{\pm} (N)} \omega_f^{\pm} (N) \frac{L'}{L}(1-c+it,f) \nonumber \\ & & \hspace{0.75in}G_+(1-c+it)\Big]i dt  \nonumber \\
& = & \frac{-1}{2\pi} \int_{\infty}^{-\infty} \Big[\sum_{f\in H_k^{\pm} (N)} \omega_f^{\pm} (N) \frac{L'}{L} (1-(c+it),f) \nonumber \\ & & \hspace{0.75in}G_+(1-(c+it)) \Big]dt.
\eea
By the functional equation $L(s,f) = \epsilon_f X_L(s) L(1-s, f)$, we have $\frac{L'}{L}(1-s,f) = \frac{X'_L}{X_L}(s) - \frac{L'}{L}(s,f)$. This gives us:
\bea
\int_{(1-c)} & = & -\frac{1}{2\pi} \int_{\infty}^{-\infty} \sum_{f\in H_k^{\pm} (N)} \omega_f^{\pm} (N) \Bigg[\frac{X'_L}{X_L}(c+it) \nonumber \\
& & \hspace{0.75in} - \frac{L'}{L}(c+it,f)\Bigg] G_-(c+it)dt  \nonumber \\
& = & \frac{1}{2\pi} \int_{-\infty}^{\infty} \left(\frac{X'_L}{X_L}(c+it)\right)G_-(c+it)dt  \nonumber \\
&   & - \frac{1}{2\pi} \int_{-\infty}^{\infty} \sum_{f\in H_k^{\pm} (N)} \omega_f^{\pm} (N) \frac{L'}{L}(c+it,f)G_-(c+it)dt. \label{ln:ints}
\eea
Let the first integral in equation \eqref{ln:ints} be denoted $\int_{X_L}$, and let the second be denoted $\int_{(c)}^*$. Then we have \be \int_{(c)}^*\ =\ \frac{1}{2\pi i} \int_{(c)} \sum_{f\in H_k^{\pm} (N)} \omega_f^{\pm}(N) \frac{L'}{L}(s,f)G_-(s)dt. \ee
Now, note that
\be D_{1, H_k^{\pm} (N); R}(\phi)\ =\ S_{1, H_k^{\pm} (N)} (g) = \int_{(c)} + \int_{(c)}^* - \int_{X_L}. \label{ln:density}\ee
By a simple contour shift and change of variables, we see that
\bea \int_{X_L} & := & \frac{1}{2\pi} \int_{-\infty}^{\infty} \frac{X'_L}{X_L}(c+it)G_-(c+it)dt \nonumber\\
                & =  & \int_{-\infty}^{\infty} \frac{X'_L}{X_L}\left(\foh+2\pi it\right)\phi(t\log R)dt. \eea
We continue to simplify this integral through the definition of $X_L$ (equation \eqref{ln:XL}), which gives the following formula
\bea & & -\int_{-\infty}^\infty \frac{X'_L}{X_L}\left(\foh + 2\pi it\right)\phi(t\log R)dt \nonumber\\ & & \ \ \ \ =\ \frac{\log N}{\log R}\hphi(0) + \frac{2}{\log R}\int_{\infty}^\infty \frac{\Gamma'}{\Gamma}\left(\frac{k}{2}+\frac{2\pi it}{\log R}\right)\phi(t)dt. \label{ln:XLint} \eea
We now evaluate $\int_{(c)} + \int_{(c)}^*$. To begin, we state a lemma from \cite{Mil5} that we use to improve the convergence of the product in the expression from Lemma \ref{lem:S_1}. We note that finding  factorizations such as the one from the following lemma is an important part of applying the $L$-functions Ratios Conjecture.

\begin{lem} \label{lem:MilProd}
Let $\Re(u) \ge 0$. Then \be \chi(u)\ :=\ \prod_p \left(1+\frac{1}{(p-1)p^u}\right)\ =\ \frac{\zeta(2)}{\zeta(2+2u)} \ \zeta(1+u)\prod_p\left(1-\frac{p^u-1}{p(p^{1+u}+1)}\right). \label{ln:chi}\ee
\end{lem}

Here we note that the product on the right hand side of the expression in the lemma converges rapidly, as each term is equal to $1 + O(1/p^2)$. We only use this lemma to note that the product on the left hand side of the expression in the lemma converges for $\Re(u) = 0$ as long as $\Im(u) \neq 0 $.

Applying this new expression for the product to the estimate from Lemma \ref{lem:S_1}, we perform the following deductions:
\bea
\int_{(c)} + \int_{(c)}^* &\ =\ & \frac{1}{2\pi i} \int_{(c)} \Big[ \sum_p \frac{\log p}{p^{2s}} \nonumber \\
& & \mp X_L(s) \frac{\zeta(2)}{\zeta(4s)}\zeta(2s) \prod_p \left(1 - \frac{p^{2s-1}-1}{p(p^{2s}+1)}\right)\Big] G(s) ds \label{ln:Mlemma} \nonumber \\
& = & \frac{1}{2\pi i} \int_{(c)} \sum_p \frac{\log p}{p^{2s}} G(s) ds \nonumber \\
& & \mp \frac{1}{2\pi i} \int_{(c)} X_L(s) \frac{\zeta(2)}{\zeta(4s)}\zeta(2s) \prod_p \left(1 - \frac{p^{2s-1}-1}{p(p^{2s}+1)}\right) G(s) ds. \ \ \ \ \  \label{ln:intc}
\eea

We thus have the following two integrals to consider:
\bea T_1 &\ :=\ & \int_{(c)} X_L(s) \frac{\zeta(2)}{\zeta(4s)}\zeta(2s) \prod_p \left(1 - \frac{p^{2s-1}-1}{p(p^{2s}+1)}\right) G(s) ds \nonumber\\
     T_2 & := & \int_{(c)} \sum_p \frac{\log p}{p^{2s}} G(s) ds. \eea
We note that the choice of subscript for $T_1, T_2$ has been made for agreement with corresponding terms in the theoretical evaluation in Section \ref{sec:THEORY}. We first determine the contribution of $T_2$. Some care is needed in its analysis, as we cannot use the Fubini-Tonelli theorem to interchange the integration and summation due to the divergence of the absolute value of the integrand.

\begin{lem} \label{lem:T2C} For $g$ satisfying $g(t) = \phi\left(\frac{t\log R}{2\pi}\right)$, we have the following expression for $T_2$:
\be \frac{1}{2 \pi i} T_2\ =\ \frac{1}{2 \pi}\sum_{p} \frac{2\log p}{p} \widehat{g}\left(\frac{2\log p}{2\pi}\right) \ = \ \frac{1}{\log R}\sum_{p} \frac{2\log p}{p} \hphi \left(\frac{2\log p}{\log R}\right). \ee
\end{lem}
\begin{proof}
We want to compute \be T_2\ =\ \int_{(c)} \left( \sum_p \frac{\log p}{p^{2s}}\right) G(s) ds \ee with $c > 1/2$. Let us write $c= \foh + \delta$, so $\delta > 0$ (and $s = c + it$). While the prime sum has a pole when $s=1/2$ (it is essentially $\zeta'(2s)/\zeta(2s)$, differing from this by a bounded factor from the sum over prime powers), this series converges absolutely when $\delta > 0$. In fact, let $X$ be an arbitrary parameter to be determined later. Then
\bea \left| \sum_{p > X} \frac{\log p}{p^{2s}} \right| & \ \le \ & \sum_{p > X} \frac{\log p}{p^{1+2\delta}} \nonumber\\
     & \ll & \sum_{p > X} \frac1{p^{1+2\delta - \epsilon}} \ \ \ ({\rm as}\ \log p \ll p^\epsilon) \nonumber\\
     & \le & \sum_{n > X} n^{-(1+2\delta - \epsilon)} \nonumber\\ & \ll & \int_X^\infty u^{-(1+2\delta-\epsilon)} du \nonumber\\
     & \ll & X^{-2\delta + \epsilon}. \eea
We thus write \be \sum_p \frac{\log p}{p^{2s}} \ = \ \sum_{p \le X} \frac{\log p}{p^{2s}} + \sum_{p > X} \frac{\log p}{p^{2s}}. \ee
We now evaluate the following two integrals:
\bea I_1 & \ := \ &  \int_{(c)} G(s) \sum_{p \le X} \frac{\log p}{p^{2s}} \nonumber\\
     I_2 & := & \int_{(c)} G(s)  \sum_{p > X} \frac{\log p}{p^{2s}}. \eea
We will change variables to replace $g$ by $\phi$, where $g(t) = \phi\left(\frac{t \log R}{2\pi}\right)$. A straightforward computation shows that $\widehat{g}(\xi) = \frac{2\pi}{\log R}
\hphi(2\pi \xi / \log R)$.

We show $I_2$ can be made arbitrarily small by choosing $X$ sufficiently large. As $c = \foh + \delta$, \be G(s) \ = \ g(t - i\delta) + g(-t + i\delta) \ = \ \phi\left( \frac{(t-i\delta)\log R}{2\pi}\right) + \phi\left( \frac{(-t+i\delta)\log R}{2\pi}\right), \ee where $R = k^2 N$ is the analytic conductor for our cuspidal newform (we will take $k$ fixed and $N\to\infty$ through the primes). Using the bound from Lemma \ref{lem:decayphi}, we find for any $n$ that
\bea \phi\left( \frac{(t-i\delta)\log R}{2\pi}\right) & \ \ll_{n,\phi} \ & \exp\left(2\pi \sigma \frac{\delta \log R}{2\pi}\right) (t^2 + (\delta/\log R)^2)^{-n} \nonumber\\
     & \ll & R^{\delta \sigma} / (t^2 + (\delta / \log R)^2)^n, \eea
where $\supp (\hphi) \subseteq (-\sigma, \sigma)$. This implies that $I_2$ can be made arbitrarily small by choosing $X$ sufficiently large:
\bea I_2 & \ \ll \ & \int_{(c)} \frac{R^{\delta \sigma}}{(t^2 + (\delta / \log R)^2)^n} \cdot X^{-2\delta + \epsilon} ds. \eea
As $ds = idt$, we see the $t$-integral converges, and is at most a power of $\log R$. We are left with the factor $R^{\delta \sigma} / X^{2\delta - \epsilon}$; if we choose $X$ large, such as $X = R^{(2011 \delta \sigma + 2011) / (2\delta-\epsilon)}$, then this piece is bounded by $R^{-1/2}$ and thus negligible; in fact, this piece tends to zero as $X\to\infty$.

We are thus left with analyzing $I_1$. Fortunately now we have a \emph{finite} prime sum. It is thus trivial to interchange the integration and summation (especially as $g$ is bounded). We now have \be I_1 \  = \ \sum_{p \le X} \log p \int_{(c)} G(s) p^{-2s}. \ee For each integral, everything is well-behaved, there are no poles, and thus we may shift the contour to $c=1/2$. This gives \be I_1 \ = \ \sum_{p \le X} \frac{\log p}{p} \int_{-\infty}^\infty (g(t) + g(-t)) p^{-2it} i dt \ = \ 2 \sum_{p \le X} \frac{\log p}{p} \int_{-\infty}^\infty g(t) p^{-2it} i dt. \ee The integral is now handled as in \cite{Mil5} (we have dropped the $1/2\pi i$ that should be outside these contour integrals; that will cancel with the $i$ here): \be \int_{-\infty}^\infty g(t)
p^{-2it}dt \ = \ \int_{-\infty}^\infty g(t) e^{-2\pi i (\frac{2\log
p}{2\pi}) t} dt \ = \ \widehat{g}\left(\frac{2\log p}{2\pi}\right)
\ee Therefore \be I_1 \ = \  2i\sum_{p \le X} \frac{\log p}{p} \widehat{g}\left(\frac{2\log p}{2\pi}\right). \ee If $X$ is sufficiently large, $\widehat{g}\left(\frac{2\log p}{2\pi}\right) = 0$ as $\widehat{g}$ has compact support. Thus if $X$ is large we may extend this sum to infinity with no error, or, equivalently, sending $X\to\infty$ means $I_2$ does not contribute and thus our original integral is just $I_1$.

Now, since $\widehat{g}(\xi) = \frac{2\pi}{\log R}\hphi(2\pi \xi / \log R)$, we have that $\frac{1}{2 \pi}\widehat{g}\left(\frac{2\log p}{2\pi}\right) = \frac{1}{\log R}\hphi\left(2\frac{\log p}{\log R}\right)$. So, we have just shown that \be \frac{1}{2 \pi i} T_2 = \frac{1}{2 \pi}\sum_{p} \frac{2\log p}{p} \widehat{g}\left(\frac{2\log p}{2\pi}\right) = \frac{1}{\log R}\sum_{p} \frac{2\log p}{p} \hphi \left(\frac{2\log p}{\log R}\right), \ee
as desired.
\end{proof}

We now consider the integral \be T_1\ :=\ \int_{(c)} X_L(s) \frac{\zeta(2)}{\zeta(4s)}\zeta(2s) \prod_p \left(1 - \frac{p^{2s-1}-1}{p(p^{2s}+1)}\right) G(s) ds. \ee
By equation \eqref{ln:chi} (which includes the definition of $\chi$), we have that \be T_1\ =\ \int_{(c)} X_L(s) \chi(2s-1) G(s) ds. \ee

To show agreement between the $L$-functions Ratios Conjecture's prediction and the theoretical evaluation of the 1-level density, we note the following:

\begin{lem} \label{lem:T1C} We have the following expression for $T_1$:
\be T_1\ =\ 4\pi i\lim_{\epsilon \downarrow 0} \int_{-\infty}^\infty  X_L\left(\foh + 2\pi i x\right) \chi(\epsilon + 4 \pi ix) \phi(t\log R) dt. \ee
\end{lem}

\begin{proof}
We begin by noting that as the only singularities in the integrand in the region of interest arise from $\chi$, and the only singularity from $\chi$ occurs at $s = \foh$, the integral is not affected by taking the limit as $c \downarrow \foh$. So
\bea T_1 &\ =\ & \lim_{c \downarrow \foh} \int_{(c)} X_L(s) \chi(2s-1) G(s) ds \nonumber \\
         & = & i\lim_{\epsilon \downarrow 0} \int_{-\infty}^\infty X_L \left(\foh + \epsilon + it\right) \chi(2\epsilon + 2it) G\left(\foh + \epsilon + it \right) dt. \eea
For a fixed value of $\epsilon$, we then shift the contour by $s \mapsto s - \epsilon$, and as this does not pass any singularities, we have:
\bea  T_1 & = & i\lim_{\epsilon \downarrow 0} \int_{-\infty}^\infty X_L \left(\foh + it\right) \chi(\epsilon + 2it) G\left(\foh + it\right) dt. \nonumber\\
          & = & 2i\lim_{\epsilon \downarrow 0} \int_{-\infty}^\infty X_L \left(\foh + it\right) \chi(\epsilon + 2it) g(t) dt. \eea
Finally, changing variables to express the integral in terms of $\phi(t) = g(2\pi t/\log R)$ gives the lemma.
\end{proof}

\begin{rek} It is important that the input to $\chi$ comes in with a factor of two, as this allows us to greatly simplify the analysis by using a simple contour shift. If the input was $\epsilon + 2it$ instead of $2\epsilon + 2it$, the result would still be true, but would require a deeper analysis.
\end{rek}

We can now prove Theorem \ref{thm:ratios}.

\begin{proof}[Proof of Theorem \ref{thm:ratios}]
Combining the expressions from equations \eqref{ln:density}, \eqref{ln:XLint}, and \eqref{ln:intc} with Lemmas \ref{lem:T2C} and \ref{lem:T1C}, we deduce Theorem \ref{thm:ratios}.
\end{proof}


\subsection{Lower Order Terms}

We now evaluate the lower order terms in the predicted 1-level density.

\begin{lem} For fixed $\delta > 0$, we have the following estimate for $T_1$:
\bea \mp \frac{1}{2\pi i}T_1
 \ = \ &\mp& \frac{1}{2}\phi(0) \pm \int_{-\infty}^{\infty} \left(\frac{\sin \left(2\pi t \frac{\log \frac{N}{4 \pi^2}}{\log R}\right)}{2\pi t}\right)\phi(t) dt \nonumber \\
     & & \mp \frac{1}{\log R}\left(2\gamma - 2\sum_p\frac{\log p}{p(p+1)} - 4\frac{\zeta'(2)}{\zeta(2)} - 2\frac{\Gamma'\left(\frac{k}{2}\right)}{\Gamma\left(\frac{k}{2}\right)}\right) \nonumber \\
     & & \cdot \ \hphi \left(\frac{\log \frac{N}{4\pi^2}}{\log R}\right) \nonumber + O\left((\log R)^{-2(1-\delta)}\right). \eea
\end{lem}

\begin{proof}
We begin by evaluating $T_1$ while ignoring the constants in front in the statement of the lemma. We observe that the infinite product in the integrand converges for $\Re (s) > 0$, and the only singularity of the integrand in the region $\Re(s) > \frac{1}{4}$ comes at $s = \foh$ from the pole of $\zeta(2s)$. In order to evaluate this integral, we shift the contour to $c = \foh$, except for a radius $\epsilon$ semi-circle around the singularity at $s = \foh$. This leaves us to evaluate
\bea & & 2i{\rm PV}\int_{-\infty}^\infty X_L\left(\foh + it\right) \frac{\zeta(2)}{\zeta(2 + 4it)}\zeta(1+2it) \prod_p \left(1 - \frac{p^{2it}-1}{p(p^{1+2it}+1)}\right) g(t) dt \nonumber\\
     & + \ & \lim_{\epsilon \downarrow 0} \int_\epsilon X_L(s) \frac{\zeta(2)}{\zeta(4s)}\zeta(2s) \prod_p \left(1 - \frac{p^{2s-1}-1}{p(p^{2s}+1)}\right) G(s) ds, \eea where ${\rm PV}$ means we take the principal value of the integral. Denote the prinicipal value integral (which is taken around $t = 0$) as $\int_P$, and the $\epsilon$ semi-circle integral as $\int_\epsilon$.

We begin by evaluating $\int_\epsilon$. As $\zeta(2s)$ has a pole of residue $\foh$ at $s = \foh$, and all of the other terms (besides $G(s)$) in the integrand take the value 1 at $s = \foh$, we see that the integrand has residue $G(\foh)/2 = g(0)$ at $s = \foh$. As the path of integration is only a semi-circle, we get half the contribution of the residue, and we deduce that \be \lim_{\epsilon \downarrow 0} \int_\epsilon X_L(s) \frac{\zeta(2)}{\zeta(4s)}\zeta(2s) \prod_p \left(1 - \frac{p^{2s-1}-1}{p(p^{2s}+1)}\right) G(s) ds \ =\ 2 \pi i \frac{g(0)}{2}\ =\ 2 \pi i \frac{\phi(0)}{2}. \ee

We now determine the contribution of $\int_P$ to the 1-level density down to an error of $O(1/\log^{2(\delta-1)} R)$. First, we change variables to express $\int_P$ in terms of $\phi(t) = g\left(t\frac{2\pi}{\log R}\right)$, giving us
\bea \int_P & \ = \ & \frac{4\pi i}{\log R} {\rm PV}\int_{-\infty}^\infty X_L\left(\foh + \frac{2\pi it}{\log R}\right) \frac{\zeta(2)}{\zeta\left(2 + \frac{8 \pi it}{\log R}\right)}\zeta\left(1+\frac{4\pi it}{\log R}\right) \nonumber\\
            & & \prod_p \left(1 - \frac{p^{4 \pi it/\log R}-1}{p(p^{1+4\pi it/\log R}+1)}\right) \phi(t) dt, \eea where PV means we take the principal value of the integral.
We now rewrite the $X_L$ term through the use of its definition: \be X_L(s) = \left(\frac{\sqrt{N}}{2 \pi}\right)^{1 - 2s} \frac{\Gamma\left(1-s + \frac{k - 1}{2}\right)}{\Gamma\left(s+\frac{k - 1}{2}\right)} \nonumber \ee
\bea X_L\left( \foh + \frac{2\pi it}{\log R} \right) & = & \left(\frac{N}{4 \pi^2}\right)^{-\frac{2\pi i t}{\log R}}\frac{\Gamma\left(\frac{-2\pi i t}{\log R} + \frac{k}{2}\right)}{\Gamma\left(\frac{2\pi i t}{\log R} + \frac{k}{2}\right)} \nonumber \\
     & = & e^{-2\pi i t \frac{\log \frac{N}{4 \pi^2}}{\log R}}\frac{\Gamma\left(\frac{-2\pi i t}{\log R} + \frac{k}{2}\right)}{\Gamma\left(\frac{2\pi i t}{\log R} + \frac{k}{2}\right)}. \eea
Note that this $\Gamma$ ratio is always of absolute value 1, as $\Gamma(\overline{z}) = \overline{\Gamma(z)}$ for $\Re (z) > 0$. Because of this, we write this ratio as $\mathfrak{G}\left(\frac{t}{\log R}\right)$.

For ease of notation, we define the function $M$ as follows:
\be M\left(\frac{t}{\log R}\right)\ :=\ \frac{\zeta(2)}{\zeta\left(2 + \frac{8 \pi it}{\log R}\right)}\ \mathfrak{G}\left(\frac{t}{\log R}\right)\prod_p \left(1 - \frac{p^{4 \pi it/\log R}-1}{p(p^{1+4\pi it/\log R}+1)}\right).  \label{ln:M}\ee

We now split the integral into two pieces which we will analyze separately:
\bea J_1 & := &  \frac{4\pi i}{\log R} \rm{PV}\int_{-(\log R)^\delta}^{(\log R)^\delta} e^{-2\pi i t \frac{\log \frac{N}{4 \pi^2}}{\log R}}M\left(\frac{t}{\log R}\right)\zeta\left(1+\frac{4\pi it}{\log R}\right) \phi(t) dt \nonumber \\
     J_2 & := & \frac{4\pi i}{\log R} \int_{|t| > (\log R)^\delta} e^{-2\pi i t \frac{\log \frac{N}{4 \pi^2}}{\log R}}M\left(\frac{t}{\log R}\right)\zeta\left(1+\frac{4\pi it}{\log R}\right) \phi(t) dt. \eea

First, we analyze $J_1$. We begin by replacing $\zeta\left(1+\frac{4\pi it}{\log R}\right)$ with just the first two terms in its Laurent expansion, $\left(\frac{\log R}{4\pi i t} + \gamma + c_1\left(\frac{4\pi it}{\log R}\right) + \cdots \right)$, where $\gamma$ is Euler's constant. By doing this, we introduce an error of size
\bea & & \frac{1}{\log R}\int_{-(\log R)^\delta}^{(\log R)^\delta} \left(c_1 \frac{t}{\log R} + c_2\frac{t^2}{(\log R)^2} + \cdots \right)O(1) dt \nonumber \\
     & \ll\ & \frac{1}{\log R} \sum_{j=1}^\infty \frac{(2011(\log R)^\delta)^{j+1}}{(\log R)^j} \nonumber \\
     & \ll\ & (\log R)^{2\delta-2} \cdot \frac{1}{1-\frac{2011(\log R)^\delta}{\log R}} \nonumber \\
     & \ll\ & (\log R)^{2(\delta-1)}. \label{ln:ZERR}\eea
We are left with determining
\be J_1^*\ =\ \rm{PV}\int_{-(\log R)^\delta}^{(\log R)^\delta} \left(\frac{1}{t} + \frac{4\pi i \gamma}{\log R}\right)e^{-2\pi i t \frac{\log \frac{N}{4 \pi^2}}{\log R}}M\left(\frac{t}{\log R}\right)\phi(t) dt. \ee

As $M\left(\frac{t}{\log R}\right)$ is analytic for $|\frac{t}{\log R}| < \frac{1}{2011}$, we take the Taylor expansion
\be M\left(\frac{t}{\log R}\right)
 =\ m_0 + m_1\left(\frac{t}{\log R}\right) + m_2\left(\frac{t}{\log R}\right)^2 + \cdots, \ee
 and note that $m_j \ll 2011^j$ as the Taylor expansion converges in $|\frac{t}{\log R}| < \frac{1}{2011}$.  Note that as we are considering only $t$ satisfying $|t| < (\log R)^\delta$, this expansion will hold over the entire region of integration if $R$ is sufficiently large. We are left to consider
\be \rm{PV}\int_{-(\log R)^\delta}^{(\log R)^\delta} \left(\frac{1}{t} + \frac{4\pi i \gamma}{\log R}\right) e^{-2\pi i t \frac{\log \frac{N}{4 \pi^2}}{\log R}}\left(m_0 + m_1\left(\frac{t}{\log R}\right) + \cdots \right)\phi(t) dt. \ee By the evenness of $\phi$ and cosine, and the fact that we are taking a principal value integral, the $m_0$ term paired with the $\frac{1}{t}$ and the cosine term from the exponential will give no contribution to the integral.

Note that $m_0 = 1$, as all factors in $M$ have value $1$ at $t=0$. Thus, the contribution from the sine term will be:
\be i\int_{-(\log R)^\delta}^{(\log R)^\delta} \frac{1}{t} \sin \left(-2\pi t \frac{\log \frac{N}{4 \pi^2}}{\log R}\right) \phi(t) dt. \ee We now note that there is a $\mp \frac{1}{2\pi i}$ outside the $T_1$. Taking this into account, this term gives
\bea & \mp & \int_{-(\log R)^\delta}^{(\log R)^\delta} \left(\frac{\sin \left(-2\pi t \frac{\log \frac{N}{4 \pi^2}}{\log R}\right)}{2\pi t}\right) \phi(t) dt \nonumber \\
     &\ =\ & \pm \int_{-\infty}^{\infty} \left(\frac{\sin \left(2\pi t \frac{\log \frac{N}{4 \pi^2}}{\log R}\right)}{2\pi t}\right)\phi(t) dt + O_A\left((\log R)^{-A}\right), \eea
for any large $A$. Note that this error term is small due to the rapid decay of $\phi$. Next, the $\frac{4\pi i \gamma}{\log R}$ term will give
\bea \frac{4\pi i \gamma}{\log R} \int_{-(\log R)^\delta}^{(\log R)^\delta} e^{-2\pi i t \frac{\log \frac{N}{4 \pi^2}}{\log R}}\phi(t) dt   \ =\  \frac{4\pi i \gamma}{\log R} \hphi \left(\frac{\log \frac{N}{4\pi^2}}{\log R}\right) + O\left((\log R)^{-A}\right),\ \ \ \ \ \eea
which, with the $\mp \frac{1}{2\pi i}$ in front of $T_1$ results in \be \mp \frac{2 \gamma}{\log R} \hphi \left(\frac{\log \frac{N}{4\pi^2}}{\log R}\right) + O\left((\log R)^{-A}\right). \ee

We now determine the contribution from $m_1$. By similar arguments to those above, pairing $\frac{4\pi i \gamma}{\log R}$ with $m_1\left(\frac{t}{\log R}\right)$ will give a term of size $O\left((\log R)^{-2(1-\delta)}\right)$. If we pair $m_1\left(\frac{t}{\log R}\right)$ with $\frac{1}{t}$, however, we get
\bea & & \int_{-(\log R)^\delta, P}^{(\log R)^\delta} \frac{1}{t}e^{-2\pi i t \frac{\log \frac{N}{4 \pi^2}}{\log R}}m_1\left(\frac{t}{\log R}\right)\phi(t) dt \nonumber \\
     & =\ & \frac{m_1}{\log R} \int_{-(\log R)^\delta}^{(\log R)^\delta}e^{-2\pi i t \frac{\log \frac{N}{4 \pi^2}}{\log R}}\phi(t) dt \nonumber \\
     & =\ & \frac{m_1}{\log R} \hphi \left(\frac{\log \frac{N}{4\pi^2}}{\log R}\right) + O\left((\log R)^{-A}\right). \eea
In Appendix \ref{ap:TC} we show, through a standard computation, that \be m_1 \ =\ -4\pi i \sum_p \frac{\log p}{p(p+1)} -8 \pi i \frac{\zeta'(2)}{\zeta(2)}-4\pi i\frac{\Gamma'}{\Gamma}\left(\frac{k}{2}\right). \ee

We now show that the remaining parts of the integrand that we have not yet considered do not contribute significantly to $J_1^*$. As the Taylor expansion converges absolutely, we can switch integration and summation. The exponential in the expression for $J_1^*$ is of size $O(1)$, so we can ignore this term in the evaluation, as we are only looking to bound above the integral of the remaining terms, and we use no cancellation in determining the bounds. Define \be S_j := \int_{-(\log R)^\delta}^{(\log R)^\delta} m_j \frac{t^{j-1}}{(\log R)^j} \phi(t)dt. \ee Then, as $\phi(t) = O(1)$, we have
\bea S_j &\ \ll\ & \left(\frac{2011}{\log R}\right)^j \int_{-(\log R)^\delta}^{(\log R)^\delta} |t^{j-1}|dt \nonumber \\
     & \ll & 2011^j(\log R)^{-j(1-\delta)}. \label{ln:SJ} \eea
Now, as the previous estimate was uniform in $j$, if $R$ is sufficiently large (so that $(\log R)^\delta > 2011$) we have \be \sum_{j = 2}^\infty S_j\ \ll\ \frac{2011^2/(\log R)^{2(1-\delta)}}{1-\frac{2011}{(\log R)^\delta}}\ \ll\ (\log R)^{2(\delta-1)}. \ee Therefore, the rest of $J_1^*$ just gives an error of size $O\left((\log R)^{2(\delta-1)}\right)$, and we have shown that
\bea \mp \frac{1}{2\pi i} J_1 = & \pm & \int_{-\infty}^{\infty} \left(\frac{\sin \left(2\pi t \frac{\log \frac{N}{4 \pi^2}}{\log R}\right)}{2\pi t}\right)\phi(t) dt \nonumber \\
     & \mp & \left(\frac{2\gamma + m_1/2\pi i}{\log R}\right)\hphi \left(\frac{\log \frac{N}{4\pi^2}}{\log R}\right) \nonumber \\
     & + & O\left((\log R)^{-2(1-\delta)}\right). \eea

Now, to show that $J_2$ is small, we use the rapid decay of $\phi$. First, however, we must estimate the terms in the integrand, and in particular, in $M(t)$. As previously explained, both the exponential and the Gamma factors of $X_L(\foh + it)$ are of size $O(1)$. Next, note that from the Dirichlet series expansion of $\zeta\left(2+\frac{8\pi i t}{\log R}\right)$, we have that \be 2 - \zeta(2)\ \leq\ \zeta\left(2+\frac{8\pi i t}{\log R}\right)\ \leq\ \zeta(2), \ee so $\frac{\zeta(2)}{\zeta\left(2 + \frac{8 \pi it}{\log R}\right)} = O(1)$. The infinite product is also $O(1)$, as each term is bounded between $1-\frac{2}{p^2-1}$ and $1+\frac{2}{p^2}$. Finally, as $L$-functions are polynomially bounded in vertical strips, let $B > 1$ so that $\zeta\left(1+ \frac{4 \pi it}{\log R}\right) \ll \frac{\log R}{t} + t^B$. Now, $\phi$ is Schwartz, we have $\phi(t) \ll t^{-(A_0)}$ for any $A_0$. Thus the entire integrand is of size $O \left( \left(\frac{\log R}{t} + t^B\right)t^{-(A_0)}\right)$. Therefore, for any $A$, we have (with an appropriate choice of $A_0$)
\bea J_2 &\ \ll\ & \int_{(\log R)^\delta}^\infty \left( \left(\frac{\log R}{t} + t^B\right)t^{-(A_0)} \right)dt \nonumber \\
     & \ll & (\log R)^{-A}. \eea
So, we have that $J_2 \ll (\log R)^{-A}$, and combining this with the estimate for $J_1$, we arrive at
\bea \mp \frac{1}{2\pi i} \int_P
\ = \ & \pm & \int_{-\infty}^{\infty} \left(\frac{\sin \left(2\pi t \frac{\log \frac{N}{4 \pi^2}}{\log R}\right)}{2\pi t}\right)\phi(t) dt \nonumber \\
     & \mp & \left(\frac{2\gamma + m_1/2\pi i}{\log R}\right)\hphi \left(\frac{\log \frac{N}{4\pi^2}}{\log R}\right) \nonumber \\
     & + & O\left((\log R)^{-2(1-\delta)}\right), \eea and so \bea \mp \frac{1}{2\pi i}T_1 = &\mp& \frac{1}{2}\phi(0) \pm \int_{-\infty}^{\infty} \left(\frac{\sin \left(2\pi t \frac{\log \frac{N}{4 \pi^2}}{\log R}\right)}{2\pi t}\right)\phi(t) dt \nonumber \\
     & \mp & \left(\frac{2\gamma + m_1/2\pi i}{\log R}\right)\hphi \left(\frac{\log \frac{N}{4\pi^2}}{\log R}\right) \nonumber \\
     & + & O\left((\log R)^{-2(1-\delta)}\right), \eea
giving the statement of the lemma.
\end{proof}

We have thus shown that
\bea \int_{(c)} + \int_{(c)}^* & = & \frac{1}{\log R}\sum_{p} \frac{2\log p}{p} \hphi \left(\frac{2\log p}{\log R}\right) \mp \frac{1}{2}\phi(0) \nonumber \\
     & & \pm \int_{-\infty}^{\infty} \left(\frac{\sin \left(2\pi t \frac{\log \frac{N}{4 \pi^2}}{\log R}\right)}{2\pi t}\right)\phi(t) dt \nonumber \\
     & & \mp \left(\frac{2\gamma + m_1/2\pi i}{\log R}\right)\int_{-\infty}^{\infty} \cos \left(2\pi t \frac{\log \frac{N}{4 \pi^2}}{\log R}\right)\phi(t) dt \nonumber \\
     & & +\ O\left((\log R)^{-2(1-\delta)}\right). \eea
     Combining this with the fact that $D_{1, H_k^{\pm} (N); R}(\phi) = S_{1, H_k^{\pm} (N)} (g)$, and that $S_{1, H_k^{\pm} (N)} = \int_{(c)} + \int_{(c)}^* - \int_{X_L}$, gives Theorem \ref{thm:LOT}.
We also note that by similar arguments, if we continue to treat individual terms in the Laurent expansions for $M(t/\log R)$ and $\zeta(1 + 4\pi it/\log R)$, as opposed to how they are treated in lines (\ref{ln:ZERR}) and (\ref{ln:SJ}), we see that for $\ell = \log (N/4\pi^2)/\log R$, we have that if $\supp(\hphi) \subseteq (-\ell,\ell)$, then all lower order terms are $\ll 1/\log^A R$ for any positive $A$. This proves Theorem \ref{thm:LOT}.

%
%

\section{Number Theory} \label{sec:THEORY}
In this section, we expand upon results from \cite{ILS} to show agreement between the $L$-functions Ratios Conjecture and theory for the family $H_k^\pm(N)$; specifically, we prove Theorem \ref{thm:sqrt}. We begin with the explicit formula from \cite{ILS} (equation 4.11):
\bea D_{1, H_k^{\pm} (N); R}(\phi) & = & \frac{\log N}{\log R}\hphi(0) + \frac{2}{\log R}\int_{-\infty}^{\infty}\frac{\Gamma'}{\Gamma}\left(\frac{k}{2}+\frac{2\pi it}{\log R}\right)\phi(t)dt - 2\sum_{f \in H_k^\pm(N)}\omega_f^\pm(N) \nonumber\\
                                   &   & \cdot \sum_p\sum_{\nu=1}^\infty\left(\alpha_f^\nu(p) + \beta_f^\nu(p)\right)\hphi\left(\frac{\nu\log p}{\log R}\right)p^{-\nu/2}\frac{\log p}{\log R}, \label{ln:EF}\eea
where $\alpha_f(p) + \beta_f(p) = \lambda_f(p)$, and $\alpha_f(p)\beta_f(p) = 1$. We now note that to determine the above quantity, we convert the sum to a sum over all $f \in H_k^*(N)$ and split by the sign of the functional equation as follows:
\bea D_{1, H_k^{\pm} (N); R}(\phi) & = & - \int_{-\infty}^{\infty}\frac{X'_L}{X_L}\left(\foh+2\pi it\right)\phi(t\log R)dt \nonumber \\ & & - 2\sum_{f \in H_k^*(N)}(1 \pm \epsilon_f)\omega_f^*(N) \nonumber \\
                                   &   & \cdot \sum_p\sum_{\nu=1}^\infty\left(\alpha_f^\nu(p) + \beta_f^\nu(p)\right)\hphi\left(\frac{\nu\log p}{\log R}\right)p^{-\nu/2}\frac{\log p}{\log R}, \label{ln:td} \eea
where $\epsilon_f = i^k\mu(N)N^{1/2}\lambda_f(N)$. We will split the sum by the factor $(1 \pm \epsilon_f)$, and consider the two pieces separately. Also, we remove the contribution from the prime $p = N$, which we can do as this term gives a contribution of size $O(N^{-1/2+\epsilon})$.

The $\hphi(0)$ piece and the $\Gamma'/\Gamma$ integral arise naturally in both the theory and the prediction of the $L$-functions Ratios Conjecture from the functional equation. For each of the two pieces arising from $(1 \pm \epsilon_f)$, we split the remaining summation into three cases: $\nu = 1$, $\nu = 2$, and $\nu \geq 3$. We will see that, for suitably restricted $\phi$, the contribution from $\nu \geq 3$ is negligible, the contribution from $\nu = 2$ corresponds to that from $T_2$, and the contribution from $\nu = 1$ corresponds to that from $T_1$.

\begin{rek} Though $H_k^*(N)$ contains only newforms, we still use the Petersson formula that involves summing over all cuspidal modular forms of weight $k$ and level $N$. This is legal because, as we are restricting the level $N$ to be prime, there are only finitely many oldforms (those of level 1), and the Petersson weights are uniform enough (see 2.52) to cause the contribution from the oldforms to be of size $O\left(\frac{1}{N^{1-\epsilon}}\right)$, which is much smaller than we hope to detect.
\end{rek}

We now include a simplified version of equation (A.8) from \cite{Mil5}, a version of the Petersson formula (see also Appendix \ref{sec:PeterssonFormula}).

\begin{lem} \label{lem:PET} For $N$ prime, with $k$ fixed, and $(m,n) = 1$, we have
\bea& &  \sum_{f \in H_k^*(N)}\omega_f^*(N)\lambda_f(m)\lambda_f(1) = \delta(m,n) + O((mN)^\epsilon/N) \nonumber \\
     & & \ \ \ \ \ \ +\ O\left(\frac{1}{N(\sqrt{(m,N) + (n,N)})}\left(\frac{mn}{\sqrt{mn} + N}\right)^{1/2}\log{2mn}\right). \eea
\end{lem}

We begin by showing the contribution from $\nu \geq 3$ is negligible. Note the following formula (for $\nu > 1$, $p \neq N$):
\be \alpha_f^\nu(p) + \beta_f^\nu(p) = \lambda_f(p^\nu) - \lambda_f(p^{\nu-2}). \label{ln:lf} \ee
With this, we can simplify the piece under consideration to:
\bea V_3 & := & \sum_{a \in \{0,1\}} \sum_{\nu=3}^\infty (\pm i^k\mu(N)N^\foh)^a \sum_{p \neq N} \hphi\left(\frac{\nu\log p}{\log R}\right)\frac{\log p}{p^{\nu/2}\log R} \nonumber\\
         &    & \cdot \sum_{f \in H_k^*(N)} \omega_f^*(N) [\lambda_f(p^\nu)-\lambda_f(p^{\nu-2})]\lambda_f(N^a), \eea where the sum over $a$ is how we split the factor $(1 \pm \epsilon_f)$.

\begin{lem} \label{lem:V3} For $\supp(\hphi) \subseteq (-\sigma,\sigma)$, we have that $V_3 \ll N^{\epsilon-3/4+\sigma/4}$.
\end{lem}
\begin{proof}
Note that as $\hphi$ has compact support ($\supp(\hphi) \subseteq (-\sigma, \sigma)$),
\be \hphi\left(\frac{\nu\log p}{\log R}\right)\frac{\log p}{p^{\nu/2}\log R} = O\left(\frac{1}{p^{\nu/2}}\right), \ee
as only primes up to $N^{\sigma+\epsilon}$ will give a nonzero value of $\hphi$ (as we will take $R$ to be a constant multiple of $N$). Thus the previous expression is bounded by
\be \sum_{\substack{a \in \{0,1\} \\ b \in \{0,2\}}} \sum_{\nu=3}^\infty \sum_{p \neq N}^{N^{\sigma/\nu+\epsilon}} \frac{1}{p^{\nu/2}}N^{a/2}\sum_{f \in H_k^*(N)} \omega_f^*(N)\lambda_f(N^a)(-1)^{b/2}\lambda_f(p^{\nu-b}). \label{ln:PET}\ee
We now note that as $p \neq N$ and $\nu - b \geq 1$, we have $(N^a,p^{\nu-b}) = 1$, and so we use Lemma \ref{lem:PET} to get that the expression from equation \eqref{ln:PET} is
\bea & \ll & \sum_{\substack{a \in \{0,1\} \\ b \in \{0,2\}}} \sum_{\nu=3}^\infty \sum_{p \neq N}^{N^{\sigma/\nu+\epsilon}} \frac{1}{p^{\nu/2}}N^{a/2}\frac{1}{N}\frac{N^\epsilon}{N^{a/2}}N^{a/4}p^{(\nu-b)/4} \nonumber \\
     & \ll & \sum_{\nu=3}^\infty \sum_{p \neq N}^{N^{\sigma/\nu+\epsilon}}\frac{1}{p^{\nu/4}}N^{\epsilon-3/4} \nonumber \\
     & \ll & (\log N)N^{\epsilon-3/4}\sum_{n=2}^{N^{\sigma/3+\epsilon}}\frac{1}{p^{\nu/4}} \ll N^{\epsilon - 3/4 + \sigma/4} \eea
\end{proof}

Thus, by ignoring the terms with $\nu \geq 3$, we introduce an error of size $O(N^{-1/2+\epsilon})$ for $\sigma < 1$, and we retain a power savings for $\sigma < 3$.

We now show agreement between $T_2$ and the $\nu = 2$ piece. By equations \eqref{ln:td} and \eqref{ln:lf}, the $\nu = 2$ piece gives the contribution $V_2 := \sum_{\substack{a \in \{0,1\} \\ b \in \{0,2\}}} S^a_b$, where
\bea S^a_b &\ :=\ & (\pm i^k\mu(N)N^\foh)^a \sum_{p \neq N} \hphi\left(\frac{2\log p}{\log R}\right)\frac{\log p}{p\log R} \nonumber\\
           &    & \cdot \sum_{f \in H_k^*(N)} \omega_f^*(N) (-1)^{b/2}\lambda_f(p^{2-b})\lambda_f(N^a). \eea

\begin{lem} \label{lem:V2} For $\supp(\hphi) \subseteq (-\sigma,\sigma)$, we have that
\be V_2\ :=\ \sum_{\substack{a \in \{0,1\} \\ b \in \{0,2\}}} S^a_b \ =\ -\sum_{p \neq N} \hphi\left(\frac{2\log p}{\log R}\right)\frac{\log p}{p\log R} + O(N^{-1/2+\epsilon} + N^{(\sigma/2)-1+\epsilon}). \ee
\end{lem}

\begin{proof}
Note that $S^0_2$ gives the contribution
\be -\sum_{p \neq N} \hphi\left(\frac{2\log p}{\log R}\right)\frac{\log p}{p\log R},\ee
which, with the constants from before, gives the exact contribution from $T_2$ up to an error of size $O(\frac{1}{N})$ (which comes from dropping $p = N$).

However, for $S^0_0$, $S^1_0$, and $S^1_2$, we are not getting diagonal terms from the Petersson formula, and again use Lemma \ref{lem:PET} to bound the contribution from these terms:
\bea S^0_0 &\ =\ & \sum_{p\neq N}^{N^{\sigma/2+\epsilon}} \frac{1}{p} \sum_{f \in H_k^*(N)} \omega_f^*(N) \lambda_f(p^2)\lambda_f(1) \nonumber\\
           & \ll & \sum_{p\neq N}^{N^{\sigma/2+\epsilon}} \frac{1}{p} \left(\left(\frac{N^\epsilon}{N}\right) + \frac{1}{N}\left(\frac{N^{\sigma/2+\epsilon}}{N^{1/2}}\right)N^\epsilon \right) \nonumber\\
           & \ll & \frac{N^\epsilon}{N}(N^\epsilon + N^{(\sigma-1)/2+\epsilon}) \ll N^{\frac{\sigma}{2}-\frac{3}{2}+\epsilon}, \eea while
\bea S^1_0 & = & N^{1/2}\sum_{p\neq N}^{N^{\sigma/2+\epsilon}} \frac{1}{p} \sum_{f \in H_k^*(N)} \omega_f^*(N) \lambda_f(p^2N)\lambda_f(1) \nonumber\\
           & \ll & N^{1/2}\sum_{p\neq N}^{N^{\sigma/2+\epsilon}} \frac{1}{p} \left(\frac{N^\epsilon}{N} + \frac{1}{N}\left(\frac{1}{N^{1/2}}\right)\left(\frac{N^{\sigma+\epsilon+1}}{N}\right)^{1/2}N^\epsilon \right) \nonumber\\
           & \ll & N^{-1/2+\epsilon} + N^{(\sigma/2)-1+\epsilon} \eea and
\bea S^1_2 & = & N^{1/2}\sum_{p\neq N}^{N^{\sigma/2+\epsilon}} \frac{1}{p} \sum_{f \in H_k^*(N)} \omega_f^*(N) \lambda_f(N)\lambda_f(1) \nonumber\\
           & \ll & N^{1/2+\epsilon} \left( \frac{N^\epsilon}{N} + \frac{1}{N} \left( \frac{1}{N^{1/2}} \right) N^\epsilon \right) \nonumber\\
           & \ll & N^{-1/2+\epsilon}. \eea
\end{proof}

We now analyze the piece from $\nu = 1$. The term in question is
\bea V_1 &\ :=\ & \sum_{a \in \{0,1\}} (i^k\mu(N)N^{1/2})^a \sum_{p \neq N} \hphi\left(\frac{\log p}{\log R}\right) p^{-1/2} \frac{\log p}{\log R} \sum_{f \in H_k^\pm(N)}\omega_f^*(N)\lambda_f(p)\lambda_f(N^a) \nonumber\\
         & = & \sum_{a \in \{0,1\}} (i^k\mu(N)N^{1/2})^a \sum_{p \neq N} \Delta_{k,N}(pN^a)\hphi\left(\frac{\log p}{\log R}\right) \frac{\log p}{\sqrt{p}\log R} \nonumber \\
         & := & \sum_{a \in \{0,1\}} P_k^a(\phi). \eea

\begin{lem} \label{lem:V1}
For $\supp(\hphi) \subseteq (-\sigma,\sigma)$, we have that
\be V_1\ =\ 2\lim_{\epsilon \downarrow 0} \int_{-\infty}^\infty \phi(x\log R) \chi(\epsilon + 4\pi ix) X_L\left(\foh + 2\pi ix\right)dx + O(N^{\sigma/2-1+\epsilon}). \ee
\end{lem}

\begin{proof}
We begin by using the Petersson formula to estimate $\Delta_{k,N}(pN^a)$, noting, as before, that there are no diagonal terms. By equation (2.8) of \cite{ILS}, we have
\be \Delta_{k,N}(pN^a)\ =\ 2\pi i^k \sum_{c \equiv 0 (N)} \frac{S(1,pN^a;c)}{c}J_{k-1}\left(\frac{4\pi\sqrt{pN^a}}{c}\right), \ee
where $S(1,pN^a;c)$ represents the classical Kloosterman sum, and $J_{k-1}$ is the Bessel function.

Following \cite{ILS}, we now make the following definition:
\be Q_k^a(1;c)\ :=\ 2\pi i^k \sum_{p \neq N} S(1,pN^a;c)J_{k-1}\left(\frac{4\pi\sqrt{pN^a}}{c}\right) \hphi\left(\frac{\log p}{\log R}\right) \frac{\log p}{\sqrt{p}\log R}. \ee
With this definition, we see that \be P^a_k(\phi)\ =\ (i^k\mu(N)N^{1/2})^a\sum_{c \equiv 0(N)}\frac{Q_k^a(1;c)}{c}. \ee
We now follow the derivation in \cite{ILS}, which uses Lemmas 6.5 and 6.6 of \cite{ILS} to reexpress $Q_k^a(1;c)$. Noting their remarks about the error involved by evaluating the Kloosterman sums for large $c$ differently (see page 98 of \cite{ILS}), we get the following:

\be Q_k^0(1;c)\ =\ 2i^k \frac{c\mu(c)}{\varphi(c)\log R}\int_0^\infty J_{k-1}(y)\hphi\left(2\frac{cy/4\pi}{\log R}\right)dy + O(N^{\sigma/2+\epsilon}(\log2c)^{-2}), \ee
and
\bea Q_k^1(1;c) &\ =\ & 2i^k\delta(1,(N,c/N)) \frac{c\mu(N)\mu^2(c/N)}{\sqrt{N}\varphi(c/N)\log R} \nonumber \\
                &   & \cdot \int_0^\infty J_{k-1}(y)\hphi\left(\frac{2\log(cy/4\pi\sqrt{N})}{\log R}\right)dy \nonumber \\
                &   & +\ O(N^{(\sigma-1)/2+\epsilon}(\log 2c)^{-2}). \eea

\begin{rek} The derivation of the above estimates for the $Q_k^a(1;c)$ terms is conditional on GRH for the Riemann Zeta function and Dirichlet $L$-functions. \end{rek}

Note that in the expression for $Q_k^0(1;c)$ the main term is absorbed into the error term. Because of this, we have that
\be P_k^0(\phi)\ \ll\ \sum_{c \equiv 0(N)}\frac{N^{\sigma/2+\epsilon}}{c(\log 2c)^2}\ = \ N^{\sigma/2-1+\epsilon}\sum_{b=1}^\infty \frac{1}{b\log^2(bN)}\ \ll\ N^{\sigma/2-1+\epsilon}. \ee
Note that a similar analysis shows that the error from $Q_k^1(1;c)$ gives an error term of size $O(N^{\sigma/2-1+\epsilon})$ to $P_k^1(\phi)$. Thus, we have that
\bea P_k^1(\phi) &\ =\ & (i^k\mu(N)N^{1/2})\sum_{c \equiv 0(N)}\frac{Q_k^1(1;c)}{c} \nonumber \\
                 & = & (i^k\mu(N)N^{1/2})\sum_{(b,N) = 1}2i^k \frac{\mu(N)\mu^2(b)}{\sqrt{N}\varphi(b)} \nonumber \\
                 &   & \cdot \int_0^\infty J_{k-1}(y)\hphi\left(\frac{2\log(by\sqrt{N}/4\pi)}{\log R}\right)\frac{dy}{\log R} + O(N^{\sigma/2-1+\epsilon}) \nonumber \\
                 & = & 2\sum_{(b,N)=1} \frac{\mu^2(b)}{\varphi(b)}\int_0^\infty J_{k-1}(y)\hphi\left(\frac{2\log(by\sqrt{N}/4\pi)}{\log R}\right)\frac{dy}{\log R} \nonumber \\
                 &   & +\ O(N^{\sigma/2-1+\epsilon}). \label{ln:JPH}\eea
Recall that $\supp(\hphi) \subseteq (-\sigma,\sigma)$. We use this fact to show that the sum over $b$ from equation \eqref{ln:JPH} converges. By the bound $J_{k-1}(x) \ll x^{k-1}$ (from equation (2.11$''$) in \cite{ILS}), it is enough to show the convergence of the sum
\be \sum_{b=1}^\infty \int_{0}^\infty y^{k-1} \hphi\left(\frac{2 \log(by\sqrt{N}/4\pi)}{\log R}\right)dy. \ee
We note that the compact support of $\hphi$ allows us to truncate the integral at $4\pi R^{\sigma/2}/b$, so what we are considering is
\be \ll\ \sum_{b=1}^\infty \int_{0}^{4\pi R^{\sigma/2}/b} y^{k-1} dy\ \ll\ \frac{1}{b^k}. \ee
As $k \geq 2$ for us, we see that the decay in $b$ is enough to give us convergence.

We now introduce a factor that will aid us by allowing us to switch the integration and summation. We have that the expression from equation \ref{ln:JPH} is equal to
\bea & = & 2\lim_{\epsilon \da 0} \sum_{(b,N)=1} \frac{\mu^2(b)}{\varphi(b)b^\epsilon}\int_0^\infty J_{k-1}(y)\hphi\left(\frac{2\log(by\sqrt{N}/4\pi)}{\log R}\right)\frac{dy}{\log R} \nonumber \\
     &   & +\ O(N^{\sigma/2-1+\epsilon}). \label{ln:JPH2}\eea

Now, using the definition of $\hphi$, and the following formula ((6.561.14) in \cite{GR})
\be \int_0^\infty J_{k-1}(y)y^s dy \ = \ 2^s\Gamma\left(\frac{k+s}{2}\right)/\Gamma\left(\frac{k-s}{2}\right), \ee
we find that the expression from equation \eqref{ln:JPH2} is equal to
\be 2\lim_{\epsilon \da 0} \sum_{(b,N) = 1} \frac{\mu^2(b)}{\varphi(b)b^\epsilon} \int_{-\infty}^\infty \phi(x\log R)\left(\frac{2\pi}{b\sqrt{N}}\right)^{4\pi ix} \frac{\Gamma\left(\frac{k}{2}-2\pi ix\right)}{\Gamma\left(\frac{k}{2}+2\pi ix\right)}dx. \label{ln:INT}\ee

Note that the introduction of $b^\epsilon$ gives rise to the ultimate similarity between the piece we are currently evaluating and the $T_1$ term from the $L$-functions Ratios Conjecture's prediction. We now define
\be \chi_N(s)\ :=\ \sum_{(b,N) = 1} \frac{\mu^2(b)}{\varphi(b)b^s}\ =\ \prod_{p \neq N}\left(1+\frac{1}{(p-1)p^s}\right). \ee

\begin{rek}
As we are following the evaluation in \cite{ILS}, we note that a function $\chi$, which serves a purpose similar to that of $\chi_N$ for us, is introduced in their exposition. We note that there is a mistake in their definition of $\chi$ that results in certain equalities being incorrect. At least in the case of $N$ prime and $k$ fixed, however, the difference between these two functions is small enough that it does not alter the main term in their analysis, which was all that was considered in that paper.
\end{rek}

Now, for any fixed $\epsilon > 0$, we can switch integration and summation (due to Tonelli's theorem), and so the main term from equation \eqref{ln:JPH} is equal to
\be 2\lim_{\epsilon \downarrow 0} \int_{-\infty}^\infty \phi(x\log R) \chi_N(\epsilon + 4\pi ix) \left(\frac{\sqrt{N}}{2\pi}\right)^{-4\pi ix}\frac{\Gamma\left(\frac{k}{2}-2\pi ix\right)}{\Gamma\left(\frac{k}{2}+2\pi ix\right)}dx. \ee
Recalling the definition of $X_L$ in equation \ref{ln:XL}, we see that this is simply
\be 2\lim_{\epsilon \downarrow 0} \int_{-\infty}^\infty \phi(x\log R) \chi_N(\epsilon + 4\pi ix) X_L\left(\foh + 2\pi ix\right)dx. \ee
As the $p=N$ factor from $\chi$ (as defined in equation \ref{ln:chi}) is of size $1 + O(1/N)$, we can replace $\chi_N$ with $\chi$ while only introducing an error of size $O(N^{-1})$, completing the proof of the lemma.
\end{proof}

\begin{proof}[Proof of Theorem \ref{thm:sqrt}]
We combine \eqref{ln:EF}, Lemmas \ref{lem:V3}, \ref{lem:V2} and \ref{lem:V1}, and compare with Theorem \ref{thm:ratios} to deduce Theorem \ref{thm:sqrt}.
\end{proof}

%
%

\appendix


\section{Petersson Formula}\label{sec:PeterssonFormula}

Below we record several useful variants of the Petersson formula. We include these versions in this paper for completeness; the material below is taken from Appendix A of \cite{Mil5}. We define \be \Delta_{k,N}(m,n) \ = \ \sum_{f \in \mathcal{B}_k(N)} \omega_f(N) \lambda_f(m) \lambda_f(n). \ee We quote the following versions of the Petersson formula from \cite{ILS} (to match notations, note that $\sqrt{\omega_f(N)} \lambda_f(n) = \psi_f(n)$).

\begin{lem}[\cite{ILS}, Proposition 2.1]\label{lem:ils21petersson} We have \be \Delta_{k,N}(m,n) \ = \ \delta(m,n) + 2\pi i^k \sum_{c\equiv 0 \bmod N} \frac{S(m,n;c)}{c} J_{k-1}\left(\frac{4\pi\sqrt{mn}}{c}\right), \ee where $\delta(m,n)$ is the Kronecker symbol, \be S(m,n;c)\ =\ \sideset{}{^*}\sum_{d \bmod c} \exp\left(2\pi i\frac{md+n\overline{d}}{c}\right)\ee is the classical Kloosterman sum ($d\overline{d} \equiv 1 \bmod c$), and $J_{k-1}(x)$ is a Bessel function.
\end{lem}

We expect the main term to arise only in the case when $m=n$ (though as shown in \cite{HuMil,ILS}, the non-diagonal terms require a sophisticated analysis for test functions with sufficiently large support). We have the following estimates.

\begin{lem}[\cite{ILS}, Corollary 2.2]\label{lem:ilscor22} We have \be \Delta_{k,N}(m,n) \ = \ \delta(m,n) + O\left(\frac{\tau(N)}{Nk^{5/6}} \ \frac{(m,n,N) \tau_3((m,n))}{\sqrt{(m,N)+(n,N)}} \ \left(\frac{mn}{\sqrt{mn} + kN}\right)^{1/2} \log 2mn\right), \ee where $\tau_3(\ell)$ denotes the corresponding divisor function (which is the sum of the cubes of the divisors of $\ell$). \end{lem}

We can significantly decrease the error term if $m$ and $n$ are small relative to $kN$.

\begin{lem}[\cite{ILS}, Corollary 2.3]\label{lem:ils23} If $12\pi\sqrt{mn} \le kN$ we have \be \Delta_{k,N}(m,n) \ = \ \delta(m,n) + O\left(\frac{\tau(N)}{2^k N^{3/2}} \ \frac{(m,n,N) \sqrt{mn}}{\sqrt{(m,N)+(n,N)}}\ \tau((m,n))\right). \ee
\end{lem}

In this paper we consider $N \to \infty$ through prime values. We must be careful. $\Delta_{k,N}(m,n)$ is defined as a sum over all cusp forms of weight $k$ and level $N$; in practice we often study the families $H_k^\sigma(N)$ of cuspidal newforms of weight $k$ and level $N$ (if $\sigma = +$ we mean the subset with even functional equation, if $\sigma = -$ we mean the subset with odd functional equation, and if $\sigma = \ast$ we mean all). Thus we should remove the contribution from the oldforms in our Petersson expansions. Fortunately this is quite easy if $N$ is prime, as then the only oldforms are those of level 1 (following \cite{ILS}, with additional work we should be able to handle $N$ square-free, though at the cost of worse error terms). We have (see (1.16) of \cite{ILS}) \be |H_k^\pm(N)| \ \sim \ \frac{k-1}{24} \ \varphi(N), \ee where $\varphi(N)$ is Euler's totient function (and thus equals $N-1$ for $N$ prime). The number of cusp forms of weight $k$ and level $1$ is (see (1.15) of \cite{ILS}) approximately $k/12$. As $\lambda_f(n) \ll \tau(n) \ll n^\gep$ and $\omega_f^\ast(N) \ll N^{-1+\gep}$, we immediately deduce

\begin{lem}\label{lem:Peterssonjustnewforms} Let $\mathcal{B}^{\rm new}_k(N)$ be a basis for $H_k^\ast(N)$ and let $\omega_f^\ast(N)$ be \be\label{eq:omegaastfN} \twocase{\omega_f^\ast(N) \ = \ }{\omega_f(1)}{if $N=1$}{\omega_f(N)/\omega(N)}{if $N>1$} \ee where $\omega(N) = \sum_f \omega_f(N)$. Note \be\label{eq:omegaastfNb} \sum_{f \in H_k^\ast(N)} \omega_f^\ast(N) \ = \ 1 \ = \ \left(1 + O\left(N^{-1+\gep}\right)\right) \sum_{f \in \mathcal{B}_k(N)} \omega_f(N). \ee For $N$ prime, we have \bea \sum_{f \in \mathcal{B}^{\rm new}_k(N)} \omega_f^\ast(N) \lambda_f(m)\lambda_f(n) & \ = \ & \Delta_{k,N}(m,n) + O\left(\frac{(mnN)^\gep k}{N}\right). \eea Substituting yields \bea & & \sum_{f \in \mathcal{B}^{\rm new}_k(N)} \omega_f^\ast(N) \lambda_f(m)\lambda_f(n) \ = \ \delta(m,n) + O\left(\frac{(mnN)^\gep k}{N}\right)  \nonumber\\ & & \ \ \ \ \ \ + \ O\left(\frac{\tau(N)}{Nk^{5/6}} \ \frac{(m,n,N) \tau_3((m,n))}{\sqrt{(m,N)+(n,N)}} \ \left(\frac{mn}{\sqrt{mn} + kN}\right)^{1/2} \log 2mn\right),\ \ \ \ \ \eea while if $12\pi\sqrt{mn} \le kN$ we have \bea & & \sum_{f \in \mathcal{B}^{\rm new}_k(N)} \omega_f^\ast(N) \lambda_f(m)\lambda_f(n) \ = \ \delta(m,n)  \nonumber\\ & & \ \ \ \ \ \ + \ O\left(\frac{\tau(N)}{2^k N^{3/2}} \ \frac{(m,n,N) \sqrt{mn}}{\sqrt{(m,N)+(n,N)}}\ \tau((m,n))\right) + O\left(\frac{(mnN)^\gep k}{N}\right). \ \ \ \ \eea
\end{lem}

\begin{proof} The proof follows by using equations \eqref{eq:omegaastfN} and \eqref{eq:omegaastfNb} in the Petersson lemmas. \end{proof}


\section{Fourier transform bound}

\begin{lem}\label{lem:decayphi}
Let $\phi$ be an even Schwartz function such that ${\rm supp}(\hphi)
\subset (-\sigma,\sigma)$. Then \be \phi(t+iy) \ \ll_{n,\phi} \ e^{2\pi |y|
\sigma} \cdot (t^2 + y^2)^{-n}. \ee
\end{lem}

\begin{proof} From the Fourier inversion formula, integrating by parts and the compact support of $\hphi$,
we have \bea \phi(t+iy) & \ = \ & \int_{-\infty}^\infty \hphi(\xi)
e^{2\pi i (t+iy)\xi} d\xi \nonumber\\ &=& \int_{-\infty}^\infty
\hphi^{(2n)}(\xi) \cdot (2\pi i (t+iy))^{-2n} e^{2\pi i
(t-iy)\xi}d\xi \nonumber\\ & \ll & e^{2\pi |y| \sigma}
(t^2+y^2)^{-n}. \eea
\end{proof}


\section{Terms involving the sign of the functional equation}\label{ap:SFE}

In this section, we treat the terms from \eqref{eq:rcpeqtwoeight} involving the sign of the functional equation. In particular, we will show that, following the other steps of the Ratios Conjecture, these terms are predicted to be quite small. Because of the nature of these terms' dependence on $N$, through a careful analysis similar to that in \cite{Mil5} (where the first of the following sums is treated -- see Remark 1.8 in \cite{Mil5}), it can be shown that the final contribution of these terms to the predicted 1-level density is of size $O(1/N)$, and so is much smaller than we could hope to detect. Rather than performing the detailed analysis as in \cite{Mil5}, we show that the only $N$-dependence in the sum is a factor of size $1/N$, which essentially implies that any contributions from this term will be of size $O(1/N)$.

We consider \bea
& & \sum_{f \in H_k^*(N)} \epsilon_f \omega_f^* (N) \sum_{h = 1}^\infty \frac{\mu_f (h)}{h^{\foh + \gamma}} X_L \left(\foh + \alpha\right) \sum_{n \le y} \frac{{\lambda}_f (n)}{n^{\foh - \alpha}}\nonumber \\
& \pm & \sum_{f \in H_k^*(N)} \epsilon_f \omega_f^* (N) \sum_{h = 1}^\infty \frac{\mu_f (h)}{h^{\foh + \gamma}} \sum_{m\le x} \frac{\lambda_f (m)}{m^{\foh + \alpha}},
\eea
where $\epsilon_f = i^k\mu(N)\lambda_f(N)\sqrt{N}.$ We will only analyze the second of the two sums, as the first is analyzed in detail in \cite{Mil5}. The analysis is similar to that contained in the proof of Theorem \ref{thm:RCP}, with a few key differences.

After replacing $\epsilon_f$ with the expression for the sign of the functional equation, we have that the second sum is \be \pm i^k\mu(N)\sqrt{N} \sum_{f \in H_k^*(N)} \omega_f^* (N) \sum_{h = 1}^\infty \frac{\mu_f (h)}{h^{\foh + \gamma}} \sum_{m\le x} \frac{\lambda_f(N) \lambda_f (m)}{m^{\foh + \alpha}}. \ee The fact that there is a $\lambda_f(N)$ in this expression is what will cause it to be small, since in following the Ratios Conjecture's recipe, we will drop the nondiagonal terms (i.e. those without a second factor of $\lambda_f(N)$).

Note that in the $m$ sum, $m$ is bounded by $x$. In the approximate functional equation, we take $x = y \sim \sqrt{N}$, and so we can conclude that $N \notdiv m$. So, as we consider only the diagonal terms, for the sum over $f \in H_k^*(N)$ to contribute a main term to the prediction, there must be another factor of $\lambda_f(N)$ arising from $\mu_f(h)$.

We now rewrite the sum as a product. Note that $\mu_f(h)$ can be defined by multiplicativity, with $\mu_f(1) = 1$, $\mu_f(p) = -\lambda_f(p)$, $\mu_f(p^2) = \chi_0(p)$ (where $\chi_0$ is the principal character to the modulus $N$), and for higher $n$, $\mu_f(p^n) = 0$. So, if there is to be any contribution from a given $h$, it must be cubefree, and in order to contribute a diagonal term, we must have $N || h$. As $N || h$, the factor for the prime $p = N$ will be $-\frac{\mu_f(N)}{N^{\foh+\alpha}}$. For a prime $p \leq x$, we have $(p,N) = 1$, so the effect of the prime could be any of  $1, \mu_f(p)\lambda_f(p)$, or $\mu_f(p^2)\lambda_f(1)$ (depending on the power of $p$ that divides $h$), so we can write the $p$ factor as $\left(1-\frac{\lambda_f(p)^2}{p^{1+\alpha+\gamma}}+\frac{1}{p^{1+2\gamma}}\right)$. As primes greater than $x$ can only arise through the $h$ sum, their factors will just have the contribution of $1$ or $\mu_f(p^2)\lambda_f(1)$, as $\mu_f(p)$ does not give a diagonal term (note that we are ignoring $p = N$, as that factor has already been determined). So the factors from $p > x$, with $p \neq N$ will be $\left(1+\frac{1}{p^{1+2\gamma}}\right)$.

So, we have just converted the sum to the product \be i^k \mu(N) \lambda_f(N) \sqrt{N} \frac{-\lambda_f(N)}{N^{\foh+\gamma}}\prod_{p \leq x} \left(1-\frac{\lambda_f(p)^2}{p^{1+\alpha+\gamma}}+\frac{1}{p^{1+2\gamma}}\right) \cdot \prod_{\substack{p > x \\ p \neq N}}\left(1+\frac{1}{p^{1+2\gamma}}\right). \ee We now replace the $\lambda_f(N)^2$ with its value, $1/N$ (as $N$ is the level of the modular form). This allows us to use the Petersson formula on the remaining terms, as they are relatively prime to $N$. We thus execute the sum over $f \in H_k^*(N)$, replacing the $\lambda_f(p)^2$ with the diagonal contribution, 1. By then extending the product over $x$ to infinity, and noting that $N$ is prime, we get \be i^k  \frac{1}{N^{1+\gamma}}\prod_{p} \left(1-\frac{1}{p^{1+\alpha+\gamma}}+\frac{1}{p^{1+2\gamma}}\right). \ee Note that the \emph{only} $N$ dependence in this product is in the $\frac{1}{N^{1+\gamma}}$ term, and as we consider only $\gamma$ with $\Re(\gamma) \geq 0$, any contribution from this factor will be of size $O(1/N)$.


\section{Taylor Coefficient of $M(t/\log R)$} \label{ap:TC}
To complete the analysis of $T_1$, we need to determine the value of the linear Taylor coefficient of $M(\frac{t}{\log R})$.

\begin{lem} The linear Taylor coefficient of $M(\frac{t}{\log R})$, as defined in equation \eqref{ln:M}, is equal to \be m_1\ =\ -4\pi i \sum_p \frac{\log p}{p(p+1)} -8 \pi i \frac{\zeta'(2)}{\zeta(2)}-4\pi i\frac{\Gamma'}{\Gamma}\left(\frac{k}{2}\right). \ee \end{lem}

In order to calculate this, we Taylor expand all the factors of $M$ in the variable $X = \frac{t}{\log R}$, where
\be M(X)\ =\ \left(\prod_p 1-\frac{p^{4\pi i X}-1}{p(p^{1+4\pi i X}+1)}\right)\left(\frac{\zeta(2)}{\zeta(2+8\pi iX)}\right)\left(\frac{\Gamma\left(-2\pi i X + \frac{k}{2}\right)}{\Gamma\left(2\pi i X + \frac{k}{2}\right)}\right). \ee
Clearly each factor (the product over primes, the $\zeta$ ratio, and the $\Gamma$ ratio) has constant term 1 in its Taylor expansion around $X = 0$, and so $m_1$ is just the sum of the linear coefficients of each of the factors. To determine these, we simply take the derivative of each factor at $X=0$. The product over primes has derivative:
\bea & & \frac{d}{dX} \prod_p \left(1 - \frac{p^{4 \pi iX}-1}{p(p^{1+4\pi iX}+1)}\right) \Big |_{X=0} \nonumber \\
     & =\ & \prod_p \left(1 - \frac{p^{4 \pi iX}-1}{p(p^{1+4\pi iX}+1)}\right) \cdot \log'\left(\prod_p \left(1 - \frac{p^{4 \pi iX}-1}{p(p^{1+4\pi iX}+1)}\right)\right) \Big |_{X=0} \nonumber \\
     & =\ &1 \cdot \sum_p \log' \left(1 - \frac{p^{4 \pi iX}-1}{p(p^{1+4\pi iX}+1)}\right) \Big |_{X=0} \nonumber \\
     & =\ & -4\pi i \sum_p \frac{\log p}{p(p+1)}. \eea

Next, the $\zeta$ ratio has derivative \be \frac{d}{dX} \frac{\zeta(2)}{\zeta(2+8\pi i X)} \Big |_{X=0}\ =\ -8\pi i \frac{\zeta'(2)}{\zeta(2)}. \ee

Finally, the $\Gamma$ ratio ($\mathfrak{G}(X)$) has derivative
\bea & & \frac{d}{dX} \frac{\Gamma\left(-2\pi i X + \frac{k}{2}\right)}{\Gamma\left(2\pi i X + \frac{k}{2}\right)} \Big |_{X=0} \nonumber \\
     & = \ & \frac{\Gamma\left(2\pi i X + \frac{k}{2}\right)\left(-2\pi i\Gamma'\left(-2\pi i X + \frac{k}{2}\right)\right)}{\Gamma\left(2\pi i X + \frac{k}{2}\right)^2} \nonumber \\
     &
       \ & - \frac{\Gamma\left(-2\pi i X + \frac{k}{2}\right)\left(2\pi i \Gamma'\left(2\pi i X + \frac{k}{2}\right)\right)}{\Gamma\left(2\pi i X + \frac{k}{2}\right)^2} \Big |_{X = 0} \nonumber \\
     & =\ & -4\pi i \frac{\Gamma'\left(\frac{k}{2}\right)}{\Gamma\left(\frac{k}{2}\right)}, \eea
giving the lemma.

\ \\

\end{document}